\documentclass[11pt]{article}
\usepackage[utf8]{inputenc}
\usepackage[T1]{fontenc}
\usepackage[english]{babel}
\usepackage[margin=2.8cm]{geometry}
\usepackage{mathpazo}
\usepackage{courier}
\usepackage{amsmath,amsfonts,amssymb,amsthm}
\usepackage{graphicx}
\usepackage{float}
\usepackage[numbers]{natbib}
\usepackage{color}

\newtheorem{rmq}{Remark}[section]  
\newtheorem{theo}{Theorem}[section]  
\newtheorem{lem}{Lemma}[section]

\DeclareUnicodeCharacter{B0}{\textdegree}

\usepackage[draft]{fixme}
\fxusetheme{color}
\FXRegisterAuthor{h}{ah}{H}
\FXRegisterAuthor{p}{ap}{P}
\FXRegisterAuthor{a}{aa}{A}

\usepackage{hyperref}
\hypersetup{
pdfpagemode=UseOutlines,      
pdfstartview=Fit,             
pdffitwindow=true,            
pdfpagelayout=TwoColumnsRight,
pdftoolbar=true,              
pdfmenubar=true,              
bookmarksopen=false,          
bookmarksnumbered=true,       
colorlinks=true,              
pdfauthor={Ton nom},     
pdftitle={Titre PDF},    
pdfcreator=PDFLaTeX,          %
pdfproducer=PDFLaTeX,         %
linkcolor=blue,               
urlcolor=blue,                
anchorcolor=blue,            
citecolor=blue,              
frenchlinks=true,             
pdfborder={0 0 0}             
}

\usepackage[usenames,dvipsnames]{xcolor}
\usepackage[colorinlistoftodos,textwidth=2.3cm]{todonotes}

\usepackage[normalem]{ulem}
\usepackage{enumitem}

\definecolor{Vert}{RGB}{0,128,0}

\begin{document}

\title{On the asymptotic rate of convergence of Stochastic Newton algorithms and their Weighted Averaged versions}
\author{Claire Boyer$^{(1,2)}$ and Antoine Godichon-Baggioni$^{(1)\star}$, \\
        $^{(1)}$ Laboratoire de Probabilités, Statistique et Modélisation\\
       Sorbonne-Université,       75005 Paris, France \\    
       $^{(2)}$ INRIA Paris \\
        \{claire.boyer or antoine.godichon$\_$baggioni\}@sorbonne-universite.fr}
\maketitle

\begin{abstract}
Most machine learning methods can be regarded as the minimization of an unavailable risk function. To optimize the latter, with samples provided in a streaming fashion, we define general  (weighted averaged) stochastic Newton algorithms, for which a theoretical analysis of their asymptotic efficiency is conducted. 
The corresponding implementations are shown not to require the inversion of a Hessian estimate at each iteration under a quite flexible framework that covers the case of linear, logistic or softmax regressions to name a few. 
Numerical experiments on simulated and real data give the empirical evidence of the pertinence of the proposed methods, which outperform popular competitors particularly in case of bad initializations. 
\end{abstract}




\section{Introduction}
Machine learning challenges are now encountered in many different scientific applications facing streaming or large amounts of data. 
First-order online algorithms have become hegemonic: by a low computational cost per iteration, they allow performing machine learning tasks on large datasets, processing each observation only once; see, for instance, the review paper \cite{bottou2018optimization}. 
Stochastic gradient methods (SGD) and their averaged versions are shown to be theoretically asymptotically efficient \cite{PolyakJud92,Pel00,GB2017}, while recent work focuses on the non-asymptotic behavior of these estimates \cite{bach2013non,godichon2016,gadat2017optimal}: more precisely, it was proven that, under mild assumptions, averaged estimates can converge at a rate of order $O(1/n)$ where we let $n$ denote the size of the dataset (and the number of iterations as well, in a streaming setting).
However, these first-order online algorithms can be shown in practice to be very sensitive to the Hessian structure of the risk they are supposed to minimize.  For example, when the spectrum of local Hessian matrices shows large variations among their eigenvalues, the stochastic gradient algorithm may be stuck far from the optimum, see for instance the application of \cite[Section 5.2]{BGBP2019}.

To address this issue, (quasi) online second-order optimization has been also considered in the literature.  
In view of avoiding highly costly iterations, most online (quasi) second-order algorithms rely on approximating the Hessian matrix by exclusively using gradient information or by assuming a diagonal structure of it (making its inversion much easier). These methods result in choosing a different step size with respect to the components of the current gradient estimate, hence the name of adaptive stochastic gradient algorithms, such as the Adagrad \cite{duchi2011adaptive} or Adadelta \cite{zeiler2012adadelta} methods. 

In this paper, we aim at minimizing the general convex function $G$ defined for any {parameter} $h\in \mathbb{R}^d$ as
$$
G(h) := \mathbb{E}\left[ g \left( X , h \right) \right],
$$
where $X$ denotes {a random variable taking value in $\mathcal{X}$ and standing for a data sample}, and $g : \mathcal{X}\times \mathbb{R}^{d} \longrightarrow \mathbb{R}$ is a loss function. The function $G$ may encode the risk of many supervised or unsupervised machine learning procedures, encompassing, for instance, linear, logistic or even softmax regressions.
Having only access to the data points $X_1,\hdots ,X_n$, i.i.d.\ copies of $X$, instead of the true underlying distribution of $X$, we propose new stochastic Newton methods to perform the optimization of $G$, relying on estimates of both the Hessian and its inverse, using second-order information of $g$. 
The reader should have noted that stochastic first-order averaged algorithms are known to be asymptotically efficient \cite{Pel00}.  Therefore, second-order methods are not expected to improve this type of theoretical guarantees but instead perform better in practice.

\subsection{Related works} 
 A Stochastic Quasi-Newton method was introduced in \cite{byrd2016stochastic}, relying on li\-mi\-ted-memory BFGS updates. Specifically, local curvature is captured through (subsampled) Hessian-vector products, instead of differences of gradients. The authors provide a stochastic Quasi-Newton algorithm which cost is close to the one of standard SGDs.
The convergence study in \cite{byrd2016stochastic} requires the boundedness from above and from below of the spectrum of the estimated Hessian inverses, uniformly over the space of parameters, which can be very restrictive. 
Besides, the framework considered in \cite{byrd2016stochastic} departs from the setting of the present paper: 
the stochastic BFGS algorithm can be seen as a refinement of mini-batches gradient algorithms, which is not explicitly derived for online purposes. 


 
In \cite{castera2019inertial}, a hybrid algorithm is proposed that combines gradient descent and Newton-like behaviors, as well as inertia.
Under the Kurdyka-Lojasiewicz (KL) property, a theoretical analysis of the associated continuous dynamical model is conducted, which significantly departs from the type of convergence guarantees established in this paper. 

A truncated Stochastic Newton algorithm has been specifically introduced for and dedicated to logistic regression in \cite{BGBP2019}.  The recursive estimates of the inverse of the Hessian are updated through the Ricatti's formula (also called the Sherman-Morri\-son's formula) leading to only $O(d^{2})$ operations at each iteration. Only in the particular case of logistic regression, optimal asymptotic behaviour of the algorithm is established under assumptions close to the ones allowed by the general framework considered in the present paper. 
Furthermore, the considered step sequence of the order of $1/n$ in \cite{BGBP2019} can freeze the dynamics of the estimates in practice, leading to poor results in the case of far initialization. 

 In \cite{LP2020}, the authors introduced
a conditioned SGD based on a preconditioning of the gradient direction. The preconditioning matrix is typically an estimate of the inverse Hessian at the optimal point, for which they obtain asymptotic optimality of the procedure under the $L$ -smoothness assumption of the objective function and boundedness assumption of the spectrum of all the preconditioning matrices used over the iterations. 
They propose to use preconditioning matrices in practice as inverse of weighted recursive estimates of the Hessian.
Therefore, the proposed conditioned SGD entails a full inversion of the estimated Hessian, requiring $O(d^3)$ operations per iteration in general, which is less compatible with large-scale data. 
Note that the weighting procedure only concerns the estimation of the Hessian, and not the whole algorithm, as we will suggest in this paper. 
The choice of the sequence of steps in the conditioned SGD remains problematic: choosing steps of order $1/n$, which is theoretically sound to obtain the optimal asymptotic behavior, may result in saturation of the algorithm far from the optimum, particularly in the case of bad initializations.

 In order to reduce the sensitivity to the initialization, an averaged Sto\-chas\-tic Gauss-Newton algorithm has been proposed in the restricted setting of non-linear regression in \cite{CGBP2020}. Despite the peculiar case of nonlinear regression, stochastic algorithms are shown to benefit from averaging in simulations.

\subsection{Contributions} In this paper, we consider a unified and general framework that includes various applications of machine learning tasks,  for which
we propose a stochastic Newton algorithm: an estimate of the Hessian is constructed and \emph{easily} updated over iterations using genuine second-order information. {Given a particular structure of the Hessian estimates that will be encountered in various applications}, this algorithm leverages from the possibility to directly update the inverse of the Hessian matrix at each iteration in $O(d^{2})$ operations, with $d$ the ambient dimension, generalizing a trick introduced in the context of logistic regression in \cite{BGBP2019}. 
For simplicity, a first version of this algorithm is studied choosing the step size of the order $O(1/n)$ where we let $n$ denote the number of iterations.  
Under suitable and standard assumptions, we establish the following asymptotic results: (i) almost sure convergence, and (ii)  almost sure rates of convergence of the iterates to the optimum, as well as (iii) a central limit theorem for the iterates.  
Nevertheless, as mentioned before, considering step sequences of order $1/n$ can lead to poor results in practice \cite{CGBP2020}. In order to alleviate this problem, we thus introduce a weighted averaged stochastic Newton  algorithm (WASNA)  which can benefit  from better step size choices and above all from weighted averaging over the iterates. We then establish the almost sure rates of convergence of its estimates, preserving the optimal asymptotic behavior of the whole procedure.
This work allows for a unified framework that encompasses the case of linear, logistic, and softmax regressions, for which WASNAs are derived, and comes with their convergence guarantees. 
To put in a nutshell, we propose new and very general weighted stochastic Newton algorithms, (i) that can be implemented efficiently in regard to their second-order characteristics, (ii) that allow various choices of weighting discussed in the light of the induced convergence optimality, and 
(iii) for which theoretical locks have been lifted
resulting in strong convergence guarantees without requiring a global strong convexity assumption.

The relevance of the proposed algorithms is illustrated in numerical experiments, challenging favorite competitors such as SGD with adaptive learning rate. 
{Even without a refined tuning of the hyperparameters, the method is shown to give good performances on the real MNIST\footnote{\url{http://yann.lecun.com/exdb/mnist/}} dataset in the context of multi-label classification.
For reproducibility purposes, the code of all the numerical experiments is available at \url{godichon.perso.math.cnrs.fr/research.html}.}


\subsection{Organization of the paper}
The general framework is presented in Section \ref{sec::framework} which introduces all the notation. The set of mild assumptions on $G$ and $g$ is also discussed.
Section \ref{sec:SN_algo} presents a new general stochastic Newton algorithm and the associated theoretical guarantees. 
Section \ref{sec:ASN_algo} introduces a new weighted averaged stochastic Newton algorithm, followed by its theoretical study in which optimal asymptotic convergence is obtained.
The versatility and relevance of the proposed algorithms is illustrated in Section \ref{sec::applications} in the case of linear, logistic, and softmax regressions { with the help of numerical implementation on simulated and real data. The theoretical guarantees for all regression settings and their associated proofs are made explicit in the appendices.}

\subsection{Notation} 
In the following, we will denote by
$\| \cdot\|$ the Euclidean norm in dimension $d$, and by
$\| \cdot\|_{op}$ the operator norm corresponding to the largest singular value in finite dimension. 
The Euclidean ball centered at $c$ and of radius $r$ will be noted as
$\mathcal{B}\left( c , r\right)$.

\section{Framework}\label{sec::framework}

Let $X$ be a random variable taking values in a space $\mathcal{X}$. The aim of this work is to estimate the minimizer of the convex function $G : \mathbb{R}^{d} \longrightarrow \mathbb{R}$ defined for all $h \in \mathbb{R}^{d}$ by
\[
G(h) := \mathbb{E}\left[ g \left( X , h \right) \right]
\]
for a loss function $g : \mathcal{X}\times \mathbb{R}^{d} \longrightarrow \mathbb{R}$. 

In this paper, under the differentiability of $G$, we assume that the first order derivatives also meet the following assumptions:

\begin{enumerate}[label=\textbf{(A\arabic*)}]
\setcounter{enumi}{0}
\item \label{(A1)} For almost every $x\in \mathcal{X}$, the functional $g(x,.)$ is differentiable  and
\begin{enumerate}[label=\textbf{(A1\alph*)}]
\item \label{(A1a)} there is $\theta \in \mathbb{R}^{d}$ such that $\nabla G( \theta) =0$;
\item \label{(A1b)} there are non-negative constants $C$ and $C'$ such that for all $h \in \mathbb{R}^{d}$,
\[
\mathbb{E}\left[ \left\| \nabla_{h} g \left( X , h \right) \right\|^{2} \right] \leq C  + C' \left( G(h) - G(\theta) \right);   
\]
\item \label{(A1c)} the functional
\[
\Sigma : h \longmapsto \mathbb{E}\left[ \nabla_{h} g \left( X,h \right) \nabla_{h} g \left( X,h \right)^{T} \right]
\]
is continuous at $\theta$.
\end{enumerate}
\end{enumerate}

Furthermore, second order information will be crucial for the definition of the Newton algorithms to come, for which we require the following assumptions to hold: 
\begin{enumerate}[label=\textbf{(A\arabic*)}]
\setcounter{enumi}{1}
\item \label{(A2)} The functional $G$ is convex, twice continuously differentiable and
\begin{enumerate}[label=\textbf{(A2\alph*)}]
\item \label{(A2a)} the Hessian of $G$ is bounded, i.e.\ there is a positive constant $L_{\nabla G}>0$ such that for all $h \in \mathbb{R}^{d}$,
\[
\left\| \nabla^{2}G (h) \right\|_{op} \leq L_{\nabla G};
\]

\item \label{(A2b)} the Hessian of $G$ is positive at $\theta$ and we denote by $\lambda_{\min}$ its smallest eigenvalue.
\item \label{(A2c)}  the Hessian of $G$ is Lipshitz on a neighborhood of $\theta$: there are positive constants $A_{\nabla^{2} G} >0$ and $L_{A,\nabla^{2} G}>0$ such that for all $h \in \mathcal{B}\left( \theta , A_{\nabla^{2} G} \right)$,
\[
\left\| \nabla^{2}G(h) - \nabla^{2}G(\theta) \right\|_{op} \leq L_{A,\nabla^{2}G} \left\| \theta - h \right\| .
\]
\end{enumerate}
\end{enumerate}

Remark that Assumption \ref{(A2a)} leads the gradient of $G$ to be Lipschitz continuous, and in particular at the optimum $\theta$, for any $h \in \mathbb{R}^{d}$, one has
\[
\left\| \nabla G  ( h) \right\| \leq L_{\nabla G} \left\| h- \theta \right\| .
\]

Overall, note that all these assumptions are very close to the ones given in \cite{pelletier1998almost}, \cite{Pel00}, \cite{GB2017} or \cite{godichon2016}. One of the main differences concerns Assumption \ref{(A1b)} in which the second order moments of $\nabla_{h}g (X,.)$ are not assumed to be upper-bounded by the squared errors $\| \cdot - \theta\|^{2}$, but by the risk error instead, i.e.\ by $G(\cdot) - G(\theta)$. Note that the first condition may entail the second one, when considering the functional $G$ to be $\mu$-strongly convex, since for any $h\in \mathbb{R}^d$,  $\| h - \theta \|^{2} \leq \frac{2}{\mu}  (G(h) - G(\theta))$. In this respect, Assumption \ref{(A1b)} can be seen as {nearly equivalent to} counterparts encountered in the literature.

\section{The stochastic Newton algorithm}
\label{sec:SN_algo}
In order to lighten the notation, let us denote the Hessian of $G$ at the optimum $\theta$ by $H = \nabla^{2}G(\theta)$. As already mentioned in \cite{BGBP2019}, the usual stochastic gradient algorithms and their averaged versions are shown to be theoretically efficient but can be very sensitive to the case where $H$ has eigenvalues at different scales. To alleviate this problem, we first focus on the stochastic Newton algorithm (SNA).

\subsection{Implementation of the SNA}

\subsubsection{The algorithm}
Consider $X_{1} , \ldots ,X_{n} , \ldots$ i.i.d.\ copies of $X$. The Stochastic Newton Algorithm (SNA) can be iteratively defined as follows: for all $n \geq 0$,
\begin{align}
\label{algo:sto_newton_iterate}
\theta_{n+1} & = \theta_{n} - \frac{1}{n+1+c_{\theta}}\overline{H}_{n}^{-1} \nabla_{h} g \left( X_{n+1} , \theta_{n} \right) ,
\end{align}
given a finite initial point $\theta_{0}$, for $c_{\theta} \geq 0$ and with $\overline{H}_{n}^{-1}$ a recursive estimate of $H^{-1}$, chosen symmetric and positive at each iteration. 
Remark that the constant $c_{\theta}$ {plays a predominant role in the first iterations (for small $n$) since it balances the impact of the first (and probably bad) estimates out. This effect decreases as the number $n$ of iterations grows and, hopefully, as the iterates improve}. 
{On a more technical side}, let us suppose that one can construct a filtration $\left( \mathcal{F}_{n} \right)$ verifying that for any $n\geq 0$,
(i) $\overline{H}_{n}^{-1}$ and $\theta_{n}$ are $\mathcal{F}_{n}$-measurable, and (ii)
 $X_{n+1}$ is independent from $\mathcal{F}_{n}$.
Note that if the estimate $\overline{H}_{n}^{-1}$ only depends on $X_{1} , \ldots ,X_{n}$, one can thus consider the filtration $\mathcal{F}_{n}$ generated by the current sample, i.e.\ for all $n \geq 1$, $\mathcal{F}_{n} = \sigma \left( X_{1} , \ldots ,X_{n} \right)$. 

\subsubsection{Construction of $\overline{H}_{n}^{-1}$}
When a natural online estimate of the Hessian $\overline{H}_{n} = (n+1)^{-1}H_{n}$ of the form
\begin{equation}
\label{form1esthessian}
H_{n} = H_{0} + \sum_{k=1}^{n} u_{k}\left( X_{k} , \theta_{k-1} \right) \Phi_{k} \left( X_{k} , \theta_{k-1} \right) \Phi_{k} \left( X_{k} , \theta_{k-1} \right)^{T} ,
\end{equation}
with $H_{0}$ symmetric and positive,
is available, a computationally-cheap estimate of its inverse can be constructed. 
Indeed, the inverse $H_{n+1}^{-1}$ can be easily updated thanks to Riccati's formula \cite{Duf97}, i.e.\ 
\begin{align}
\label{eq:riccatti_trick}
H_{n+1}^{-1} = H_{n}^{-1} - u_{n+1} \left( 1+ u_{n+1} \Phi_{n+1}^{T}H_{n}^{-1}\Phi_{n+1} \right)^{-1} H_{n}^{-1} \Phi_{n+1} \Phi_{n+1}^{T}H_{n}^{-1}
\end{align}
with $\Phi_{n+1} =\Phi_{n+1} \left( X_{n+1} , \theta_{n} \right)$ and $u_{n+1} = u_{n+1} \left( X_{n+1} , \theta_{n} \right) $. In such a case, one can consider the filtration generated by the sample again, i.e.\ $\mathcal{F}_{n} = \sigma \left( X_{1} , \ldots ,X_{n} \right)$.
In {Appendix \ref{sec::app::theoretical}}, the construction of the recursive estimates of the inverse of the Hessian will be made explicit in the cases of linear, logistic and softmax regressions.

\subsection{Convergence results for the SNA}
\label{sec:SN_convergence_results}

In this part, we derive theoretical guarantees about consistency, rates of convergence and asymptotic efficiency of the estimates $(\theta_n)_n$ defined in \eqref{algo:sto_newton_iterate}. These kinds of results can be inter-weaved as follows: 
first, one should obtain the consistency of the estimates $(\theta_n)_n$, which, coupled with the convergence of $(\overline{H}_n)_n$ (see Assumption \ref{(H2a)} coming), entails a.s.\ rates of convergence for $(\theta_n)_n$; the latter can be then combined with convergence rates for $\left(\overline{H}_n\right)_n$ (see Assumption \ref{(H2b)} coming), to establish a central limit theorem for $(\theta_n)_n$.  In this part, we derive theoretical guarantees about consistency, rates of convergence and asymptotic efficiency of the estimates $(\theta_n)_n$ defined in \eqref{algo:sto_newton_iterate}. These kinds of results can be inter-weaved as follows: 
first, one should obtain the consistency of the estimates $(\theta_n)_n$, which, coupled with the convergence of $(\overline{H}_n)_n$ (see Assumption \ref{(H2a)} coming), entails a.s.\ rates of convergence for $(\theta_n)_n$; the latter can be then combined with convergence rates for $\left(\overline{H}_n\right)_n$ (see Assumption \ref{(H2b)} coming), to establish a central limit theorem for $(\theta_n)_n$. {Remark that the consistency and the rate of convergence of the estimates of the Hessian are usually  deduce from the ones of the estimates of $\theta$.}

A usual key ingredient to establish the almost sure convergence of the estimates constructed by  stochastic algorithms is the Robbins-Siegmund theorem \cite{Duf97}. 
To do so for the stochastic Newton algorithm proposed in \eqref{algo:sto_newton_iterate}, we need to prevent the possible divergence of the eigenvalues of $\overline{H}_{n}^{-1}$. Furthermore, to ensure convergence to the true parameter solution $\theta$, 
the smallest eigenvalue of $\overline{H}_{n}^{-1}$ should be bounded from below, that is, $\lambda_{\max} \left( \overline{H}_{n} \right)$ should be bounded from above.
To this end, we require the following assumption to be satisfied:
\begin{enumerate}[label=\textbf{(H\arabic*)}]
\setcounter{enumi}{0}
\item \label{(H1)} The largest eigenvalue of the Hessian estimates $\overline{H}_{n}$ and its inverse $\overline{H}_{n}^{-1}$ are controlled: there is $\beta \in (0,1/2)$ such that
\[
\lambda_{\max} \left( \overline{H}_{n} \right) = O (1) \quad a.s \quad \quad \text{and} \quad \quad \lambda_{\max} \left( \overline{H}_{n}^{-1} \right) = O \left( n^{\beta} \right) \quad a.s.
\]
\end{enumerate}
If verifying Assumption \ref{(H1)} in practice may seem difficult, we will see in  {Appendix \ref{sec::app::theoretical} (Equations \eqref{alg:WASNA_log}, \eqref{alg:SNA_soft} or \eqref{alg:WASNA_soft})} how to modify the recursive estimates of the inverse the Hessian to obtain such a control on their spectrum, while preserving a suitable filtration.
The following theorem gives the strong consistency of the stochastic Newton estimates constructed in \eqref{algo:sto_newton_iterate}.
\begin{theo}\label{snaas}
 Under Assumptions \ref{(A1a)}, \ref{(A1b)}, \ref{(A2a)} \ref{(A2b)} and \ref{(H1)}, the iterates of the stochastic Newton algorithm given in \eqref{algo:sto_newton_iterate} satisfy 
\[
\theta_{n} \xrightarrow[n\to + \infty]{a.s. } \theta .
\]
\end{theo}
The proof is given in Appendix \ref{app:proof_as} and consists in a particular case of the proof of the almost sure convergence for a more general algorithm than the SNA.

The convergence rate of the iterates for the SNA is a more delicate result to obtain, that requires the convergence of $\overline{H}_{n}$ and $\overline{H}_{n}^{-1}$. This is why we suppose  from now on that the following assumption is fulfilled:
\begin{enumerate}[label=\textbf{(H2\alph*)}]
\setcounter{enumi}{0}
\item \label{(H2a)} {$\overline{H}_{n}$ converges almost surely to $H$.}
\end{enumerate}
Assumption \ref{(H2a)} {is generally entailed by the almost sure convergence of $\theta_{n}$ (one can look for instance at the proof of Theorem \ref{theo:ASNsoft} in the setting of the softmax regression to see how such an assumption is verified)}.  Remark that Assumption \ref{(H2a)} also implies the convergence of $\overline{H}_{n}^{-1}$. 
The following theorem gives the rate of convergence associated to the SNA {iterates in}  \eqref{algo:sto_newton_iterate}.

\begin{theo}\label{theoratetheta}
Under Assumptions \ref{(A1)}, \ref{(A2)}, \ref{(H1)} and \ref{(H2a)}, the iterates of the stochastic Newton algorithm given in \eqref{algo:sto_newton_iterate} satisfy 
for all $\delta > 0$,
\[
\left\| \theta_{n} - \theta \right\|^2 = o \left( \frac{(\ln n)^{1+\delta}}{n} \right) \quad a.s.
\]
In addition, if there exist positive constants $a> 2$ and $C_{a}$ such that for all $h \in \mathbb{R}^{2}$,
\begin{equation}
\label{momenta}\mathbb{E}\left[ \left\| \nabla_{h} g (X,h) \right\|^{a} \right] \leq C_{a} \left( 1+ \left\| h- \theta \right\|^{a} \right),
\end{equation}
then
\[
\left\| \theta_{n} - \theta \right\|^{2} = O \left( \frac{\ln n}{n} \right) \quad a.s.
\]
\end{theo}
The proof of Theorem \ref{theoratetheta} is given in Appendix \ref{app:proof_rate_SN}.
This type of results are analogous to the usual ones dedicated to online estimation based on Robbins-Monro procedures \cite{PolyakJud92,Pel00,godichon2016}. 
Besides, one could note that the requirements on the functional $G$ to obtain rates of convergence for the SNA are very close to the ones requested to obtain rates of convergence for the averaged stochastic gradient descent \cite{Pel00}. 

Refining the theoretical analysis of SNA defined in \eqref{algo:sto_newton_iterate}, we now aim at studying the variance optimality of such a procedure.
To establish strong results as the asymptotic efficiency of the parameter estimates, the estimates of the Hessian should admit a (weak) rate of convergence. In this sense, we consider the following assumption:
\begin{enumerate}[label=\textbf{(H2\alph*)}]
\setcounter{enumi}{1}
\item \label{(H2b)} {There exists a positive constant $p_{H}$ such that}
\[
\left\| \overline{H}_{n} - H \right\|^{2} = O \left( \frac{1}{n^{p_{H}}} \right) \quad a.s.
\]
\end{enumerate}

 {Note that Assumption \ref{(H2b)} can be obtained when a rate of convergence for $\theta_{n}$ is available.}

 {All the ingredients are now there for establishing} the optimal asymptotic normality of the  {SNA} iterates \eqref{algo:sto_newton_iterate}.

\begin{theo}\label{thetatlc}
Under Assumptions \ref{(A1)},\ref{(A2)}, \ref{(H1)}, \ref{(H2a)} and  \ref{(H2b)},  the iterates of the stochastic Newton algorithm given in \eqref{algo:sto_newton_iterate} satisfy 
\[
\sqrt{n} \left( \theta_{n} - \theta \right) \xrightarrow[n\to + \infty]{\mathcal{L}} \mathcal{N}\left( 0 , H^{-1} \Sigma H^{-1} \right) ,
\]
with $\Sigma = \Sigma (\theta) :=  \mathbb{E}\left[ \nabla_{h} g (X , \theta ) \nabla_{h} g (X , \theta )^{T} \right] $.
\end{theo}
The proof is given in Appendix \ref{app:proof_tlc_SN}. 
In Theorem \ref{thetatlc}, the SNA estimates \eqref{algo:sto_newton_iterate} are ensured to be asymptotically efficient provided usual assumptions on the functional $G$ matching the ones made in \cite{PolyakJud92,Pel00,GB2017}.
This result highlights the benefit of stochastic Newton algorithms over standard online gradient methods, which have been shown not to be asymptotically optimal \cite{pelletier1998almost}, unless considering an averaged version \cite{Pel00}.

Note that this result has been achieved independently of the work of \cite{LP2020}.

The central limit theorem established in Theorem \ref{thetatlc} requires convergence rates of the Hessian estimates by Assumption \ref{(H2b)}. This can been often ensured at the price of technical results  (such as the Lipschitz continuity of the Hessian or a quite large number of bounded moments of $X$).
Note that Assumption \ref{(H2b)} is required in the proof to make the rest terms (see \eqref{dectheta}) negligible, and let the martingale term govern the rate of convergence. 
Relaxing such an assumption in view of obtaining similar theoretical guarantees should be doable but requires to consequently modify the proof strategy.   
In Section \ref{subsec::convwasn}, we will see how to relax Assumption \ref{(H2b)} for a modified version of the stochastic Newton algorithm.
\section{The weighted averaged stochastic Newton algorithm}

\label{sec:ASN_algo}

In this section, we propose a modified version of the SNA, by allowing non-uniform or uniform averaging over the iterates. This leads to the weighted averaged stochastic Newton algorithm (WASNA). To our knowledge, this is the first time that {
a general WASNA covering a large spectrum of applications is studied, allowing non-uniform weighting as well.
}

\subsection{Implementation of the WASNA}
\label{subsec:ASN_algo}

\subsubsection{The algorithm}
As mentioned in \cite{CGBP2020}, considering the stochastic Newton algorithm ends up taking decreasing steps at the rate $1/n$ (up to a matrix multiplication), which can clog up the dynamics of the algorithm. A bad initialization can then become a real obstacle to the high performance of the method.  
To circumvent this issue, we consider the Weighted Averaged Stochastic Newton Algorithm (WASNA) defined recursively for all $n\geq 0$ by
\begin{align}
\label{def::thetatilde}\tilde{\theta}_{n+1} & = \tilde{\theta}_{n} - \gamma_{n+1} \overline{S}_{n}^{-1}\nabla_{h} g \left( X_{n+1} , \tilde{\theta}_{n} \right) \\
\label{def::genmoy} \theta_{n+1,\tau} & = \left( 1- \tau_{n+1} \right)\theta_{n,\tau} + \tau_{n+1}\tilde{\theta}_{n+1} ,
\end{align}
given 
\begin{itemize}
    \item finite starting points $\theta_{\tau,0} = \tilde{\theta}_{0}$,
    \item $\gamma_{n} = \frac{c_{\gamma}}{\left(n+c_{\gamma}'\right)^{\gamma}}$ with $c_{\gamma}> 0$, $c_{\gamma}' \geq 0$ and $\gamma \in (1/2,1)$,
    \item $\overline{S}_{n}^{-1}$ a recursive estimate of $H^{-1}$, chosen symmetric and positive at each iteration,
    \item the weighted averaging sequence  $\left( \tau_{n} \right)$ that should satisfy
\begin{itemize}
\item $\left( \tau_{n} \right)$ is $\mathcal{G}\mathcal{S} ( {-1}) $  (see \cite{mokkadem2011generalization}), i.e.\ 
\begin{align}
\label{eq:weight_seq_cond1}
n \left( 1- \frac{\tau_{n-1}}{\tau_{n}} \right) \xrightarrow[n\to + \infty]{} -1 .
\end{align}
\item There is a constant 
$\tau >  {1/2}$ such that
\begin{align}
\label{eq:weight_seq_cond2}
n\tau_{n} \xrightarrow[n\to + \infty]{} \tau .
\end{align}
\end{itemize}
\end{itemize}  


 { By choosing different sequences $(\tau_{n})_n$, one can play more or less on the strength of the last iterates in the optimization.
This strategy can be indeed motivated by limiting the effect of bad initialization of the algorithms.
For instance, choosing $\tau_{n} = \frac{1}{n+1}$ (which is compatible with \eqref{eq:weight_seq_cond1} and \eqref{eq:weight_seq_cond2}) leads to the "usual" averaging in stochastic algorithms (see \citep{CGBP2020} for instance for an averaged version of a stochastic Newton algorithm specific to the non-linear regression setting).
This approach can be generalized considering $\tau_{n} = \frac{\log(n+1)^{\omega}}{\sum_{k=0}^{n} \log(k+1)^{\omega}}$ with $\omega > 0
$ which leads to a Weighted Averaged Stochastic Newton Algorithm.
For both averaging, since $\tau =1$, one has the following unified iterates
\begin{equation}\label{def::thetan1}
\theta_{n,1} = \frac{1}{\sum_{k=0}^{n}\log(k+1)^{\omega}}\sum_{k=0}^{n}\log(k+1)^{\omega}\tilde{\theta}_k,
\end{equation}
where choosing $\omega =0$ leads to the "usual" averaging technique.
}

 {In addition, note that the only difference between $\overline{H}_{n}^{-1}$ and $\overline{S}_{n}^{-1}$ is that they respectively depend on $(\theta_n)_n$ defined in \eqref{algo:sto_newton_iterate} and  $(\theta_{n,\tau})_n$ (or eventually $(\tilde{\theta}_{n})_n$) given in \eqref{def::genmoy}.
{Here again, if natural Hessian estimates admit the form \eqref{form1esthessian}, the Riccatti's trick used in \eqref{eq:riccatti_trick} can be applied to directly update $\overline{S}_{n}^{-1}$.}
 {In Appendix \ref{sec::app::theoretical}}, we exemplify the construction of recursive estimates $(\overline{S}_{n}^{-1})_n$ for the Hessian inverse in the cases of linear, logistic and softmax regressions.}



\subsection{Convergence results}
\label{subsec::convwasn}



 {As in the case of the SNA, we need to prevent the possible divergence of the eigenvalues of $\overline{S}_{n}$ and $\overline{S}_{n}^{-1}$ to ensure the convergence of the WASNA estimates \eqref{def::thetatilde} and \eqref{def::genmoy}. 
}

\begin{theo}\label{theothetanbar}
Under Assumptions \ref{(A1a)}, \ref{(A1b)}, \ref{(A2a)}, \ref{(A2b)},  { suppose in addition that $\overline{S}_{n}$ verifies \ref{(H1)} with $\beta < \gamma -1/2$}.  The iterates of the weighted averaged stochastic Newton algorithm given in \eqref{def::thetatilde} and \eqref{def::genmoy} therefore satisfy
\[
\tilde{\theta}_{n} \xrightarrow[n\to + \infty]{a.s.\ } \theta \quad \quad \quad \quad \text{and} \quad \quad \quad \quad {\theta}_{n,\tau} \xrightarrow[n\to + \infty]{a.s.\ } \theta .
\]
\end{theo}
The proof is given in Appendix \ref{app:proof_as}. 
As expected, averaging does not introduce bias in the estimates.
 {
Similarly to the SNA, the strong consistency of the Hessian estimates is required to derive the rate of convergence of the WASNA estimates, as given in the following theorem.
}
\begin{theo}\label{theothetatilde}
Under Assumptions \ref{(A1)}, \ref{(A2a)}  {and} \ref{(A2b)}, suppose also that  {$\overline{S}_{n}$ converges almost surely to $H$ and that} there are positive constants $\eta >   \frac{1}{\gamma}-1$ 
and $C_{\eta}$ such that
\begin{equation}
\label{momenteta} \mathbb{E}\left[ \left\| \nabla_{h}g \left( X , h \right) \right\|^{2+2\eta} \right] \leq C_{\eta} \left( 1+ \left\| h - \theta \right\|^{2+2\eta} \right) .
\end{equation}
Then the iterates of the weighted averaged stochastic Newton algorithm given in \eqref{def::thetatilde} verify that
\[
\left\| \tilde{\theta }_{n} - \theta \right\|^{2} = O \left( \frac{\ln n}{n^{\gamma}} \right) \quad a.s.
\]
\end{theo}

The proof is given in Appendix \ref{app:proof_theothetatilde}. 
Theorem \ref{theothetatilde} involves a standard extra assumption in \eqref{momenteta}, which matches the usual ones for stochastic gradient algorithms \cite{pelletier1998almost,GB2017}.
Note that this theoretical result could be adapted to the case of an inconsistent estimate $S_n$ of $H$, at the price of extra technicalities. However, in practice, the consistency of the estimates $\theta_{n,\tau}$ (given by Theorem \ref{theothetanbar}) quasi-systematically leads to the consistency of the Hessian estimates, see the different examples developed in Section \ref{sec::app::theoretical}.

\begin{theo}\label{theomoy}
Under Assumptions \ref{(A1)},\ref{(A2a)},\ref{(A2b)} and Inequality \eqref{momenteta},  suppose in addition that $\overline{S}_{n}$ verifies \ref{(H2b)} vith $p_{H} > \frac{1- \gamma}{2}$.  Then the iterates of the weighted averaged stochastic Newton algorithm defined in \eqref{def::genmoy} satisfy
\[
\left\| \theta_{n,\tau} - \theta \right\|^{2} = O \left( \frac{\ln n}{n} \right) \quad a.s.
\]
and
\[
\sqrt{n} \left( \theta_{n,\tau } - \theta \right) \xrightarrow[n\to + \infty]{\mathcal{L}}  \mathcal{N}\left( 0 , \frac{\tau^{2}}{2\tau -1} H^{-1} \Sigma H^{-1} \right)
\]
with $\Sigma := \Sigma (\theta) = \mathbb{E}\left[ \nabla_{h} g \left( X,\theta \right) \nabla_{h} \left( X, \theta \right)^{T} \right]$.
\end{theo}
The proof is given in Appendix \ref{app:proof_thm_theomoy}. 
One could note that Assumption {\ref{(H2b)}} is required not only to establish the asymptotic normality of the iterates, but also to get the rate of convergence of the WASNA iterates defined in \eqref{def::genmoy}. This was not the case  for the SNA.

\begin{rmq}[On the asymptotic optimality of weighted versions]
{
Particularizing Theorem \ref{theomoy} to the weighted averaged stochastic Newton algorithm defined by \eqref{def::thetan1} gives
\[
\sqrt{n} \left( {\theta}_{n,1} - \theta \right) \xrightarrow[n\to + \infty]{\mathcal{L}} \mathcal{N}\left( 0 , H^{-1} \Sigma H^{-1} \right) 
\]
i.e these estimates are asymptotically efficient for both considered averaging methods (the usual and the logarithmic ones). 
Beware, the choice of weights is crucial as it can damage the asymptotic variance. 
Indeed, using another non-uniform weighting of the form $\theta_{n,1+\omega} =   \frac{1}{\sum_{k=0}^{n}  (k+1)^{\omega}} \sum_{k=0}^{n}   (k+1)^{\omega} \tilde{\theta}_{k}  $ leads to
\[
\sqrt{n} \left( \theta_{n,1+\omega} - \theta \right) \xrightarrow[n\to + \infty]{\mathcal{L}} \mathcal{N}\left( 0 , \frac{( 1+ \omega)^{2}}{2\omega +1} H^{-1}\Sigma H^{-1} \right).
\]
}
\end{rmq}

If  {Assumption \ref{(H2b)}} may represent a theoretical lock for the application of Theorem \ref{theomoy}, it can be by-passed: 
the idea is to exploit a particular structure of the Hessian estimates, to derive the rate of convergence and the asymptotic normality of the WASNA iterates. This is the purpose of the following theorem.
\begin{theo}\label{jenaipleinle}
Suppose that the Hessian estimates $(\overline{S}_{n})_n$ in the WASNA iterate \eqref{def::thetatilde} is of the form
\[
\overline{S}_{n} = \frac{1}{n+1} \left( S_{0} + \sum_{k=1}^{n} \overline{u}_{k} \overline{\Phi}_{k} \overline{\Phi}_{k}^{T} + \sum_{k=1}^{n} \frac{c_{\beta}}{k^{\beta}}Z_{k}Z_{k}^{T} \right)
\]
with $S_{0}$ a symmetric and positive matrix, $c_{\beta}\geq 0$, $\beta\in (\gamma -1/2)$ and
\begin{align*}
& \overline{u}_{k} = u_{k} \left( X_{k} , \theta_{\tau,k-1} \right) \in \mathbb{R} & \overline{\Phi}_{k} = \Phi_{k} \left( X_{k} ,  \theta_{\tau,k-1} \right)\in \mathbb{R}^d,
\end{align*}
and $(Z_{k})_k$ standard Gaussian vectors in dimension $d$.
Assume that
\begin{itemize}
\item for all $\delta > 0$, there is a positive constant $C_{\delta}$ such that for all $k$, 
\begin{equation}
\label{majinthess}\mathbb{E}\left[ \left\| \overline{u}_{k} \overline{\Phi}_{k} \overline{\Phi}_{k}^{T} \right\| \mathbf{1}_{\left\lbrace \left\| \theta_{\tau,k-1} - \theta \right\| \leq (\ln k)^{1/2+ \delta}\sqrt{\gamma}_{k} \right\rbrace} |\mathcal{F}_{k-1}\right]  \leq C_{\delta},
\end{equation}
\item there is $\alpha \in (1/2 , \tau)$ and $\delta > 0$ such that
\begin{align}
\notag
\label{sumfini} & \sum_{k\geq 0} (k+1)^{2\alpha}\frac{\tau_{k+1}^{2}}{\gamma_{k+1}}\frac{(\ln k)^{1+\delta}}{(k+1)^{2}} \\ 
& \times \mathbb{E}\left[  \left\| \overline{u}_{k} \overline{\Phi}_{k} \overline{\Phi}_{k}^{T} \right\|^{2} \mathbf{1}_{\left\lbrace \left\| \theta_{\tau,k-1} - \theta \right\| \leq (\ln k)^{1/2+ \delta}\sqrt{\gamma}_{k} \right\rbrace} |\mathcal{F}_{k-1}\right] < + \infty \quad a.s.\
\end{align}
\end{itemize}
Under the additional Assumptions \ref{(A1)}, \ref{(A2a)}, \ref{(A2b)} and  \eqref{momenteta}, if $\overline{S}_{n}$ converges almost surely to $H$, the iterates of the weighted averaged stochastic Newton algorithm defined in \eqref{def::genmoy} satisfy
\[
\left\| \theta_{n,\tau} - \theta \right\|^{2} = O \left( \frac{\ln n}{n} \right) \quad a.s.\
\]
and
\[
\sqrt{n} \left( \theta_{n,\tau} - \theta \right) \xrightarrow[n\to + \infty]{} \mathcal{N} \left( 0 , \frac{\tau^{2}}{2\tau -1 } H^{-1}\Sigma H^{-1} \right)
\]
with $\Sigma := \mathbb{E}\left[ \nabla_{h} g \left( X, \theta \right) \nabla_{h} g \left( X, \theta \right)^{T} \right] $.
\end{theo}
The proof is given in Appendix \ref{app:proof_thm_jenaipleinle}. 
The main interest of this theorem is that it enables to get the asymptotic normality of the estimates without having the rate of convergence of the estimates of the Hessian, at the price of a special structure of the latter. More specifically, contrary to \cite{BGBP2019}, no Lipschitz assumption on the functional 
$$
h \longmapsto \mathbb{E}\left[ u_{k} \left( X_{k} , h \right) \Phi_{k} \left( X_{k}, h \right) \Phi_{k}\left( X_{k}, h \right)^{T} |\mathcal{F}_{k-1}\right]
$$
is needed. Note in particular that \eqref{sumfini} is verified for all $\alpha < \frac{3-\gamma}{2}$ since  $$\mathbb{E}\left[  \left\| \overline{u}_{k} \overline{\Phi}_{k} \overline{\Phi}_{k}^{T} \right\|^{2} \mathbf{1}_{\left\lbrace \left\| \theta_{\tau,k-1} - \theta \right\| \leq (\ln (k))^{1/2+ \delta}\sqrt{\gamma}_{k} \right\rbrace} |\mathcal{F}_{k-1}\right] $$ 
is uniformly bounded.

{Theorem \ref{jenaipleinle} will be particularly useful for the coming practical applications, since this special structure in the Hessian estimates is met in the case of linear, logistic and softmax regressions.}
\section{Applications}\label{sec::applications}

 In this section, we assess the numerical performance of the weighted stochastic Newton algorithm 
 and compare it to that of second-order online algorithms {(whose different properties are summarized in Table \ref{tab::resume})}:
\begin{itemize}
    \item the stochastic Newton algorithm {(SNA)} defined in \eqref{algo:sto_newton_iterate} with a step in $1/n$, similar to the one studied in \cite{BGBP2019} specifically for the logistic regression;
    \item the stochastic Newton algorithm {(SNA)} defined in \eqref{algo:sto_newton_iterate} with a step in $1/n^{3/4}$;
    \item the {weighted} averaged stochastic Newton algorithm {(WASNA)} given in \eqref{def::thetatilde} and \eqref{def::genmoy}, with standard weighting ($\tau_{n} = 1/(n+1)$);
    \item the weighted averaged stochastic Newton algorithm (log-WASNA) given in \eqref{def::thetatilde} and \eqref{def::genmoy} with logarithmic weighting ($\tau_{n} = \frac{\log(n+1)^{\omega}}{\sum_{k=0}^{n} \log(k+1)^{\omega}}$ and $\omega =2$); 
\end{itemize}
with first-order online methods:
\begin{itemize}
    \item the stochastic gradient algorithm (SGD) \cite{pelletier1998almost} with step $1/n^{3/4}$;
    \item the averaged Stochastic Gradient Algorithm (ASGD) \cite{bach2013non};
\end{itemize}
and finally with first-order online algorithms mimicking second-order ones:
\begin{itemize}
    \item the Adagrad algorithm \cite{duchi2011adaptive}, which uses adaptive step sizes using only first-order information,
    \item the averaged Adagrad algorithm, with standard weighting.
\end{itemize}
We illustrate their performance in three different learning tasks, the case of (i) linear, (ii) logistic, and (iii) softmax regressions, for simple and more complex structured input data.

{The SNA and WASNA algorithms are made explicit for each application in the appendices. Furthermore, we provide them in a formal way with their associated theoretical guarantees.  }

\begin{table}[h!]
\begin{tabular}{lccccc}
\hline
Method             & Stepsize        & Averaging & Weights                                                        & \begin{tabular}[c]{@{}c@{}}Limit of the \\ conditioning\\ matrix\end{tabular} & \multicolumn{1}{l}{\begin{tabular}[c]{@{}l@{}}Asymptotic\\ \\ efficiency\end{tabular}} \\ \hline
SGD                & $n^{-\gamma}$   & No        & No                                                              & $I_{d}$                                                                       & No                                                                                     \\
ASGD               & $n^{-\gamma}$   & Yes       & $\frac{1}{n+1}$                                                & $I_{d}$                                                                       & Yes                                                                                    \\
Adagrad            & $n^{-\gamma}$   & No        & No                                                              & $\left( \text{Diag} (\Sigma ) \right)^{-1/2}$                                 & No                                                                                     \\
Adagrad (averaged) & $n^{-\gamma}$   & Yes       & $\frac{1}{n+1}$                                                & $\left( \text{Diag} (\Sigma ) \right)^{-1/2}$                                 & Yes                                                                                    \\
SNA                & $\frac{1}{n+1}$ & No        & No                                                             & $H^{-1}$                                                                      & Yes                                                                                    \\
SNA                & $n^{-\gamma}$   & No        & No                                                             & $H^{-1}$                                                                      & No                                                                                     \\
WASNA (standard)   & $n^{-\gamma}$   & Yes       & $\frac{1}{n+1}$                                                & $H^{-1}$                                                                      & Yes                                                                                    \\
log-WASNA          & $n^{-\gamma}$   & Yes       & $\frac{\log(n+1)^{\omega}}{\sum_{k=0}^{n} \log(k+1)^{\omega}}$ & $H^{-1}$                                                                      & Yes                                                                                   
\end{tabular}
\caption{\label{tab::resume} Summary of stochastic algorithms and their related properties. Note that the conditioning matrix corresponds to the matrix involved in the product with the gradient at each step.}
\end{table}

\subsection{Least-square regression}
 Consider the least square regression model defined by
\[
\forall n \geq 1, \quad Y_{n} = X_{n}^{T}\theta + \epsilon_{n} 
\]
where $X_{n} \in \mathbb{R}^{d}$ is a random feature vector, $\theta \in \mathbb{R}^{d}$ is the parameter vector, $\epsilon_{n}$ is a zero-mean random variable independent from $X_{n}$, and $\left( X_{i} , Y_{i} , \epsilon_{i} \right)_{i\geq 1}$ are independent and identically distributed.
{
Then, the model parameter $\theta$ can be seen as the minimizer of the functional $G : \mathbb{R}^{d} \longrightarrow \mathbb{R}_{+}$ defined for all parameter $h \in \mathbb{R}^{p}$ by
\[
G(h) = \frac{1}{2}\mathbb{E}\left[ \left( Y - X^{T}h \right)^{2} \right] .
\]}
{Implementation details for both stochastic Newton algorithms proposed in this paper for this specific setting are discussed in Appendix \ref{subsec:LS}.} 
We now fix $d=10$, and we choose $$\theta = (-4,-3,2,1,0,1,2,3,4,5)^{T} \in \mathbb{R}^{10},$$ 
$X \sim \mathcal{N}\left( 0 , \text{diag}\left( \frac{i^{2}}{d^{2}}\right)_{i=1,\ldots , 10} \right)$ and $\epsilon \sim \mathcal{N}(0,1)$. Note that in such a case the Hessian associated to this model is equal to $\text{diag}\left( \frac{i^{2}}{d^{2}}\right)_{i=1,\ldots , 10}$, meaning that the largest eigenvalue is $100$ times larger than the smallest one. Therefore, considering stochastic gradient estimates leads to a step sequence which cannot be adapted to each direction. In Figure \ref{fig:ls_diagCov}, we monitor the quadratic mean error of the different estimates, for three different type of initializations.  
One can see that both averaged Newton methods and the stochastic Newton method with step size of the order $1/n$ outperform all the other algorithms, specially for far intializations (right). The faster convergence of Newton methods or of the Adagrad algorithm compared to the one of standard SGD can be explained by their ability to manage the diagonal structure of the Hessian matrix, with eigenvalues at different scales. 
\begin{figure}[h]
\centering\includegraphics[width=0.95\textwidth]{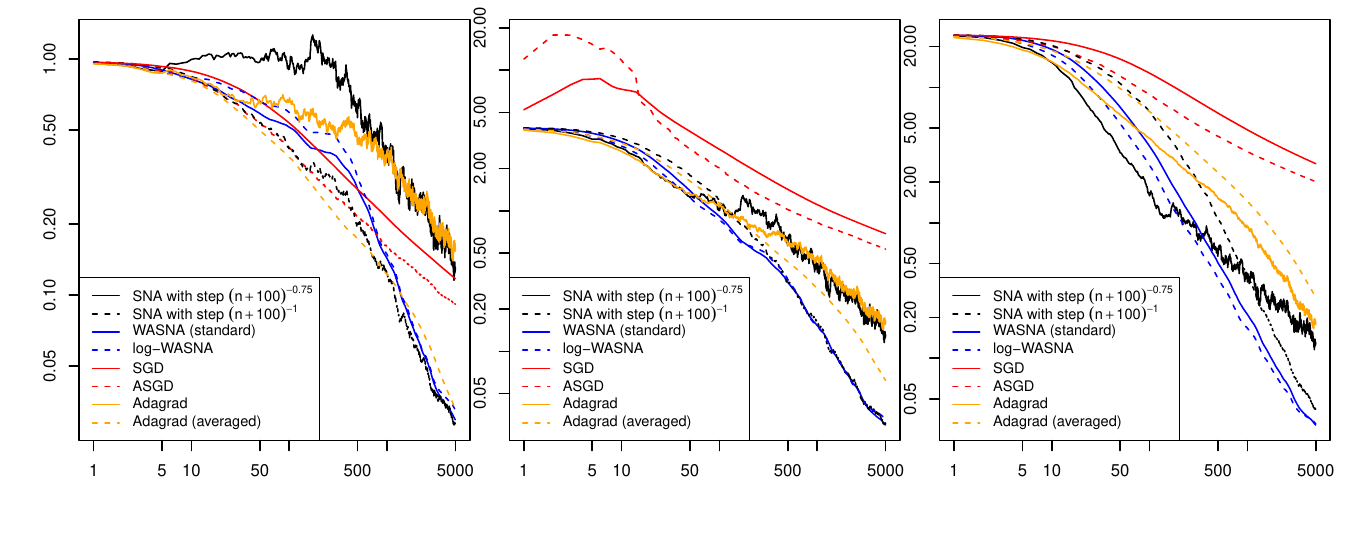}
\caption{\label{fig:ls_diagCov} 
(Linear regression with uncorrelated variables) Mean-squared error of the distance to the optimum $\theta$ with respect to the sample size for different initializations: $\theta_{0} = \theta + r_{0}U$, where $U$ is a uniform random variable on the unit sphere of $\mathbb{R}^{d}$ with  $r_{0} = 1$ (left), $r_0=2$ (middle) or $r_0=5$ (right). Each curve is obtained by an average over $50$ different samples (drawing a different initial point each time).
}
\end{figure}
\begin{figure}[h]
\centering\includegraphics[width=0.99\textwidth]{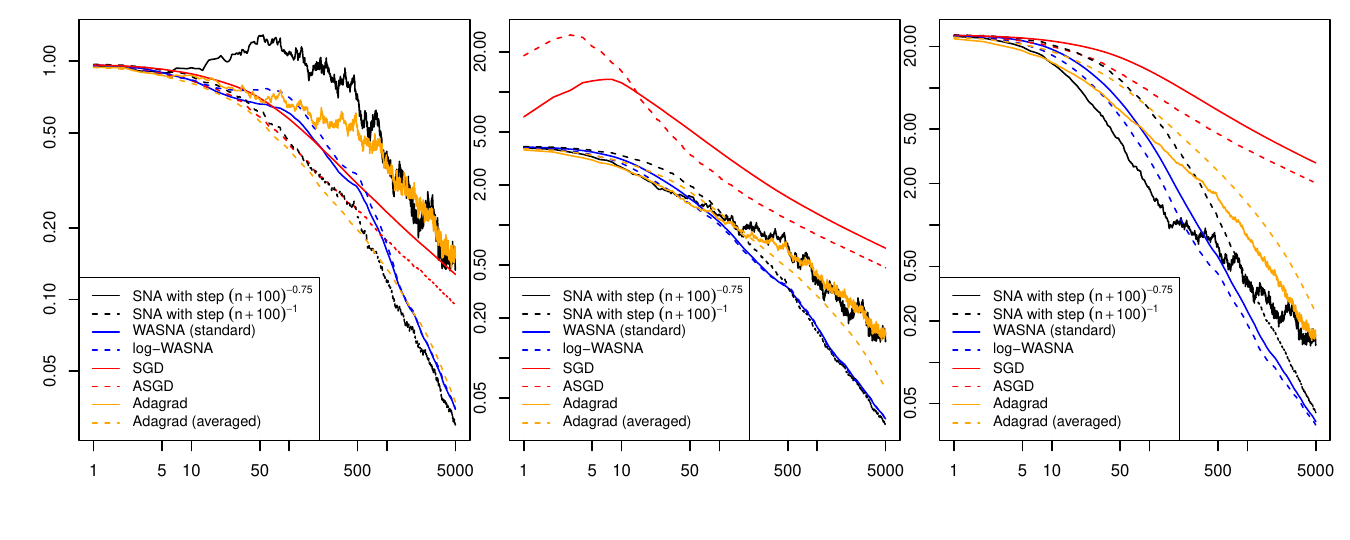}
\caption{\label{fig:ls_fullCov}
(Linear regression with correlated variables) Mean-squared error of the distance to the optimum $\theta$ with respect to the sample size for different initializations: $\theta_{0} = \theta + r_{0}U$, where $U$ is a uniform random variable on the unit sphere of $\mathbb{R}^{d}$, with  $r_{0} = 1$ (left), $r_0=2$ (middle) or $r_0=5$ (right). Each curve is obtained by an average over $50$ different samples (drawing a different initial point each time).}
\end{figure}

Consider now a more complex covariance structure of the data, such as follows 
$$X \sim \mathcal{N}\left( 0 , A \text{diag}\left( \frac{i^{2}}{d^{2}} \right)_{i=1 , \ldots , d} A^{T} \right)$$ 
where $A$ is a random orthogonal matrix. This particular choice of the covariates distribution, by the action of $A$, allows to consider strong correlations between the coordinates of $X$.
In Figure \ref{fig:ls_fullCov}, one can notice that the choice of adaptive step size used in the Adagrad algorithm is no longer sufficient to give the best convergence result in the presence of highly correlated data. In such a case, both averaged Newton algorithms remarkably perform, showing their ability to handle complex second-order structure of the data, and all the more so for bad initializations (right).

\subsection{Logistic regression} 
We turn out to be a binary classification problem: the logistic regression model defined by
\[
\forall n \geq 1 , \quad Y_{n} | X_{n}  \sim \mathcal{B}\left( \pi \left( \theta^{T}X \right) \right) 
\] 
where $X_{1} , \ldots ,X_{n} , \ldots$ are independent and identically distributed random vectors lying in $\mathbb{R}^{d}$ and for all $x \in \mathbb{R}$, 
\[
\pi (x) = \frac{\exp (x)}{1+ \exp (x)} .
\]
{
Then, the model parameter $\theta$ can be seen as the minimizer of the functional $G : \mathbb{R}^{d} \longrightarrow \mathbb{R}_{+}$ defined for all parameter $h \in \mathbb{R}^{p}$ by
\[
G(h) = \mathbb{E}\left[ \log\left( 1 + \exp \left(  X^\top h\right)\right) - X^{\top} h Y \right] .
\]} 
Due to the intrinsic non-linear feature of this model, this is not clear how the covariance structure of the covariates may affect the training phase, clearly departing from the linear regression setting. {The details about the implementation of the stochastic Newton algorithms in this specific framework, as well as their associated theoretical results are given in Appendix \ref{subsec:logisticRegression}.}
The first logistic regression setting that we consider is inspired from \cite{BGBP2019} and defined by
the model parameter $\theta = (9,0,3,9,4,9, 15, 0,7, 1, 0)^{T} \in \mathbb{R}^{11}$, with an intercept and standard Gaussian input variables, i.e.\  $X = \left(1 , \Phi^{T} \right)^{T}$ with $\Phi \sim \mathcal{N}\left( 0,I_{10} \right)$. 
In Figure \ref{fig:logistic_regression1}, we display the evolution of the {quadratic mean error} of the different estimates, for three different initializations. 
The Newton methods converge faster, in terms of distance to the optimum, than online gradient descents, which can be again explained by the Hessian structure of the model: even if the features are standard Gaussian random variables, the non-linearity introduced by the logistic model leads to a covariance structure difficult to apprehend theoretically and numerically by first-order online algorithms. 
In case of bad initialization (right), the best performances are given by the averaged Adagrad algorithm, the stochastic Newton algorithm with steps in $O(1/n^{3/4})$, closely followed by averaged stochastic Newton algorithms.
For the latter, this optimal asymptotic behaviour is enabled in particular by the use of weights giving more importance to the last estimates (being compared to the "standard" averaging Newton algorithm).
One can see that in such an example the step choice for the non-averaged Newton algorithm is crucial: choosing a step sequence of the form $1/n$ as in \cite{BGBP2019} significantly slows down the optimization dynamics, whereas a step choice in $O(1/n^{3/4})$ allows to reach the optimum much more quickly.

\begin{figure}[h]
\centering\includegraphics[width=0.99\textwidth]{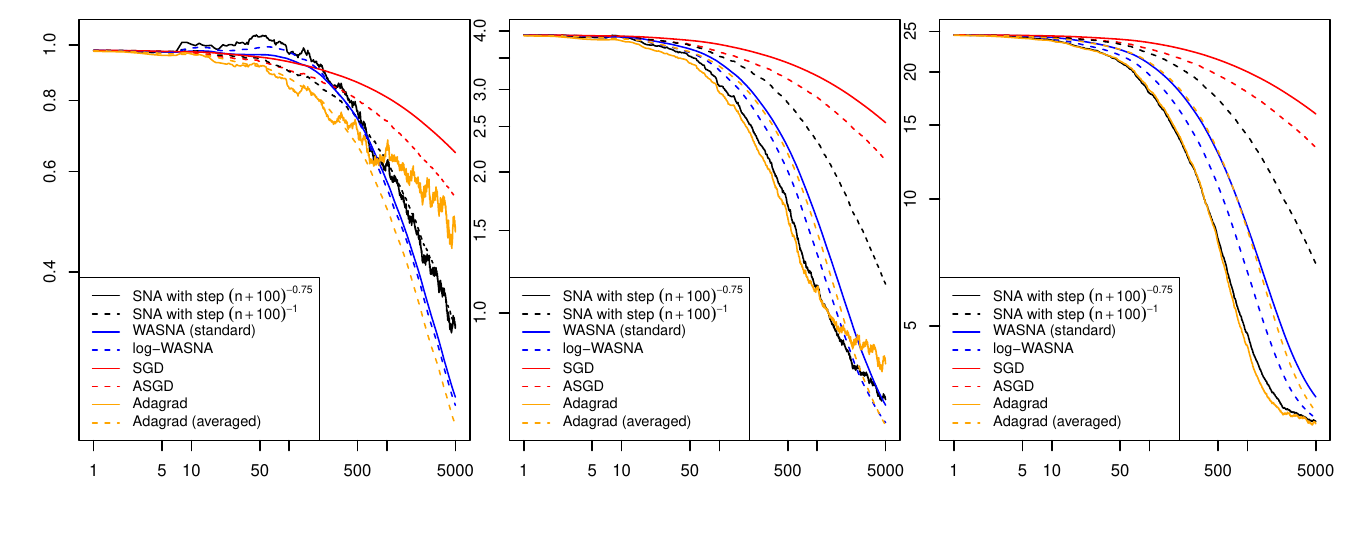}
\caption{\label{fig:logistic_regression1}(Logistic regression with standard Gaussian variables) Mean-squared error of the distance to the optimum $\theta$ with respect to the sample size for different initializations: $\theta_{0} = \theta + r_{0}U$, where $U$ is a uniform random variable on the unit sphere of $\mathbb{R}^{d}$, with  $r_{0} = 1$ (left), $r_0=2$ (middle) or $r_0=5$ (right). Each curve is obtained by an average over $50$ different samples (drawing a different initial point each time).}
\end{figure}



Let us now consider a second example, consisting in choosing $\theta \in \mathbb{R}^{d}$ with all components equal to $1$, and $X \sim \mathcal{N}\left( 0 , A \text{diag}\left( \frac{i^{2}}{d^{2}} \right)_{i=1 , \ldots , d} A^{T} \right)$ where $A$ is a random orthogonal matrix. The results are displayed in Figure \ref{figlogreg3}.
In the presence of such highly structured data, the averaged stochastic Newton algorithms are shown to perform greatly, even more for initializations far from the optimum (middle and right). 
One could note that in all initial configurations (left to right), the stochastic Newton algorithm with step size in $O(1/n)$ proves to be a worthy competitor, whereas the Adagrad algorithm becomes less and less relevant as the starting point moves away from the solution.  

\begin{figure}[h]
\centering\includegraphics[width=0.99\textwidth]{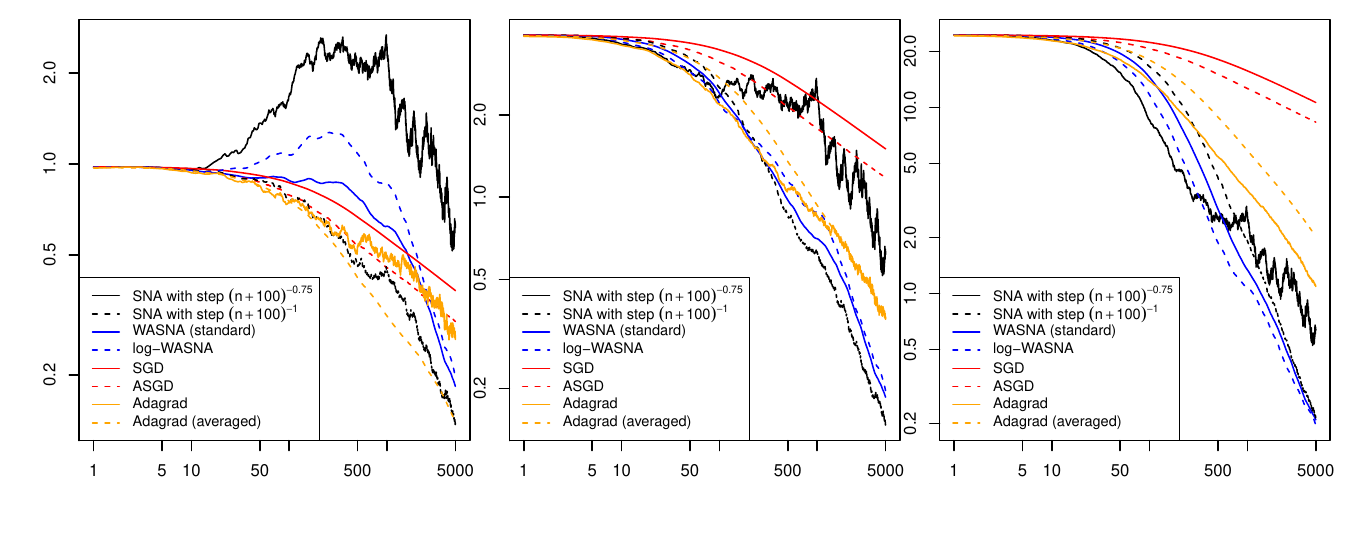}
\caption{\label{figlogreg3}
(Logistic regression with correlated Gaussian variables) Mean-squared error of the distance to the optimum $\theta$ with respect to the sample size for different initializations: $\theta_{0} = \theta + r_{0}U$, where $U$ is a uniform random variable on the unit sphere of $\mathbb{R}^{d}$, with  $r_{0} = 1$ (left), $r_0=2$ (middle) or $r_0=5$ (right). Each curve is obtained by an average over $50$ different samples (drawing a different initial point each time).}
\end{figure}


\subsection{Softmax regression}\label{secsoft}

In this section, we focus on softmax regression, which consists of a multilabel classification case.
Assume that the number of different classes is $K$, and that the model parameters are  $\theta_{1} \in \mathbb{R}^{p} , \ldots , \theta_{K} \in \mathbb{R}^p$. Consider the samples $\left( X_{1} , Y_{1} \right) ,$ 
$\ldots , \left( X_{n} , Y_{n} \right), \ldots$ being i.i.d.\ random vectors in $\mathbb{R}^d \times [1,\hdots , K]$ with for all $n \geq 1$ and $k = 1 , \ldots , K$,
\[
\mathbb{P}\left[ Y_{n} = k | X_{n} \right] = \frac{e^{\theta_{k}^{T}X_{n}}}{\sum_{k' = 1}^{K} e^{\theta_{k'}^{T}X_{n}}}.
\]
 {
In this framework, the model parameter $(\theta_1, \hdots \theta_K)$ can be seen as the minimizer of the functional $G : \mathbb{R}^{d} \times \ldots \times \mathbb{R}^{d} \longrightarrow \mathbb{R}$ defined for all $h=(h_1,\hdots , h_K)\in \mathbb{R}^d\times \hdots \times \mathbb{R}^d$ by
\[
G(h) = -\mathbb{E} \left[ \log \left( \frac{e^{h_{Y}^{T}X}}{\sum_{k' = 1}^{K} e^{h_{k'}^{T}X}}  \right) \right],
\]
where $\left( X,Y \right)$ is a i.i.d.\ copy of  $\left( X_{1} , Y_{1} \right) $.}
 {The stochastic Newton algorithms for the softmax regression are detailed in Appendix \ref{secsoft} as well as their paired theoretical results.}
We have performed numerical experiments by running the SNA and the WASNA with different tuning, as well as the previous considered benchmark on simulated data. As the conclusions are similar to the ones obtained in the case of the logistic regression, we discuss them in Appendix \ref{app:softmax_simu}.

We focus instead on the  MNIST\footnote{\url{http://yann.lecun.com/exdb/mnist/}} real dataset, in order to illustrate the performance of the WASNA in a context of multi-label classification. 
It consists in $70000$ pictures of $28\times 28$ pixels representing handwritten digits recast into vectors of dimension $784$. 
The goal is to predict the digit $Y\in \{0, \hdots, 9\}$ represented on each vectorized image $X \in \mathbb{R}^{784}$, 
where each coordinate gives the contrast (between $0$ and $255$) of each pixel. 
This is then a multi-label classification setting with 10 different classes.
In a preprocessing step, we normalize the features between 0 and 1 before applying the softmax regression. More formally, the model can be defined for any $k \in \{0 , \ldots 9\}$ by
\[
\mathbb{P}\left[ Y = k | X \right] = \frac{e^{\theta_{k}^{T}X}}{\sum_{k=0}^{9}e^{\theta_{k}^{T}X}}
\]
with the parameters $\theta= \left( \theta_{0}^{T} , \ldots, \theta_{9}^{T} \right)^{T}$ and the normalized features $X \in [0,1]^{784}$. 
Despite the simplicity of this model which is not really adapted to imaging data, the obtained performances are noteworthy, even when applied directly on the raw pixels data. 
The dataset is randomly split into a training set of size $60000$ and a test set of size $10000$ and the WASNA is run with \emph{default} parameters, i.e.\ $\gamma = 0.75$, $c_{\gamma} = 1$, $c_\gamma ' = 0$ and $\omega =2$ on the training set. The constructed estimates of the parameter $\theta$ are then used to "read" the digit displayed in pictures of the test set, leading to an overall performance of $88 \%$ accurate predictions. For completeness, and to understand which digits are mistaken, we provide the resulting confusion matrix in Figure \ref{fig:confmatrix}.
\begin{figure}
    \centering
    \includegraphics[height=8cm]{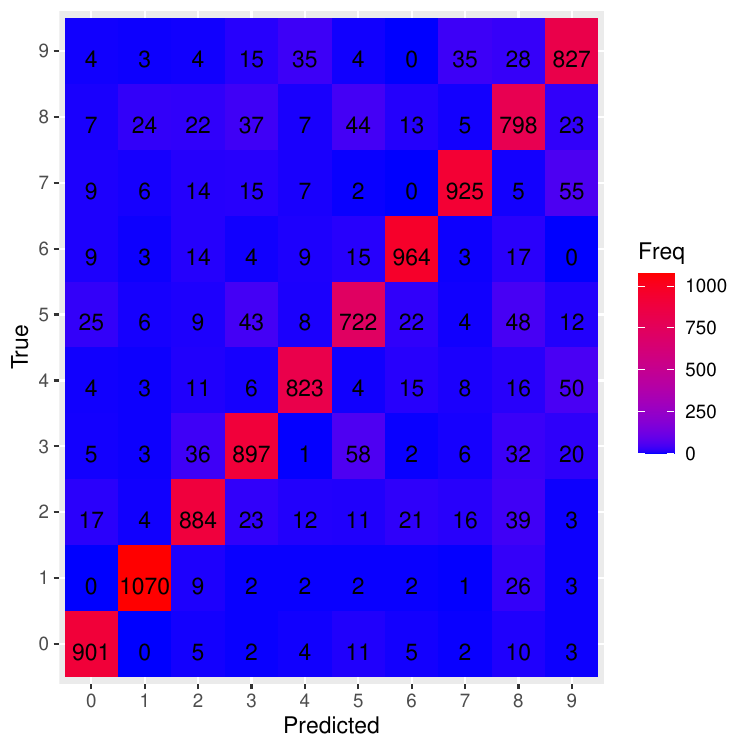}
    \caption{(Softmax regression on the MNIST dataset) Confusion matrix for the predictions given by the default WASNA on a test set of size 10000. }
    \label{fig:confmatrix}
\end{figure}
Remark that Averaged Stochastic gradient algorithms and the Adagrad one leads to analogous (or slightly better) results. The comparison in terms of accuracy may not be totally fair as the hyperparameters of the WASNA have not been optimized at all but chosen as default values.  
This numerical experiment on the MNIST real dataset proves the proposed WASNA to be a second order online method able to tackle large-scale data. 
And if the number of hyperparameters can be a legitimate delicate point raised by some readers, it should be noted that a default choice however already leads to very good results on a difficult problem such as the classification of images into 10 classes. 



\section{Conclusion}
In this paper, we have provided a unified framework for deriving stochastic Newton algorithms and their new averaged versions, paving the way for genuine second-order online algorithms for machine learning.
Under mild assumptions, we have established convergence results, such as almost sure convergence rates and asymptotic efficiency of the constructed estimates.
  {The WASNA allows in particular non-uniform averaging techniques, to give more weights to the last iterates. If the usual and logarithmic weightings are comparable in terms of theoretical guarantees (i.e.\ asymptotic efficiency), their performance can differ on applications. 
  On top of that, both versions require the calibration of several additional hyperparameters, which can be seen as a limitation for the implementation. 
  One can, however, note that in the numerical experiments, the arbitrary use of logarithmic weights (with $\omega=2$ gives the best practical results in most of the cases, providing more stability and less sensitivity to bad initializations. }
  
 {Furthermore, to be totally fair, the comparisons with other online methods have been done only regarding the sample size and not considering the computation time (which is of course more expensive for Newton's method).}
The Riccatti's trick to directly update the Hessian inverse estimates is also very beneficial in practice to reduce the iteration cost usually attributed to second-order methods. 
The next steps to explore are: (i) finding a surrogate trick for the Ricatti's formula in order to extend the algorithms to more general loss functions; (ii) refining Newton algorithms 
to optimize the storage needed by such algorithms, that could become a lock for  very high-dimensional data.

\subsection*{Current Funding Sources}
There are currently no Funding Sources.

\subsection*{Data availability statements}
All the codes are available there: \url{http://godichon.perso.math.cnrs.fr/codes_newton.zip}.  MNIST data are available at this link \url{http://yann.lecun.com/exdb/mnist/}.

\subsection*{Conflict of interest}
The authors declare that they have no known competing financial interests or personal relationships that
could have appeared to influence the work reported in this paper.

\bibliographystyle{plain}

\bibliography{biblio_redaction_2}

\appendix
\section{Applications: algorithms and theoretical results}\label{sec::app::theoretical}

\subsection{Least-square regression}
\label{subsec:LS}
Let us recall the least squares regression model defined by
\[
\forall n \geq 1, \quad Y_{n} = X_{n}^{T}\theta + \epsilon_{n} 
\]
where $X_{n} \in \mathbb{R}^{d}$ is a random features vector, $\theta \in \mathbb{R}^{d}$ is the parameters vector, and $\epsilon_{n}$ is a zero-mean random variable independent from $X_{n}$, and $\left( X_{i} , Y_{i} , \epsilon_{i} \right)_{i\geq 1}$ are independent and identically distributed. Then, $\theta$ can be seen as the minimizer of the functional $G : \mathbb{R}^{d} \longrightarrow \mathbb{R}_{+}$ defined for all parameter $h \in \mathbb{R}^{p}$ by
\[
G(h) = \frac{1}{2}\mathbb{E}\left[ \left( Y - X^{T}h \right)^{2} \right] .
\]
The Hessian of $G$ at $h$ is given by $\nabla^{2}G(h) = \mathbb{E}\left[ XX^{T} \right]$ and a natural estimate is
\[
\overline{H}_{n} = \frac{1}{n}\sum_{i=1}^{n} X_{i}X_{i}^{T}
\]
whose inverse can be easily updated thanks to the Riccati's formula leading to
\[
H_{n+1}^{-1} = H_{n}^{-1} - \left( 1+ X_{n+1}^{T}H_{n}^{-1}X_{n+1} \right)^{-1} H_{n}^{-1} X_{n+1}X_{n+1}^{T}H_{n}^{-1}
\]
where $H_{n}=(n+1)\overline{H}_{n}$.
 Then, assuming that $H=\mathbb{E}\left[ XX^{T} \right]$ is positive and that $X$ and $\epsilon$ admits 4-th order moments, entails that all assumptions of Section \ref{sec::framework} are verified, and by the law of the iterated logarithm, 
 \[
 \left\| \overline{H}_{n} - H \right\|^{2} = O \left( \frac{\ln \ln n}{n} \right) \quad a.s.
 \]
 Then the convergence results of the Stochastic Newton algorithm in Section \ref{sec:SN_convergence_results} hold. Furthermore, if there is $\eta > 0$ verifying $\eta >  \frac{1}{\gamma}-1$ and such that $X$ and $\epsilon$ admit moment of order $4 + 4\eta$, the convergence results of the averaged version in Section \ref{subsec::convwasn} hold.

\subsection{Logistic regression} \label{subsec:logisticRegression}

Let us recall that the logistic regression model is defined by
\[
\forall n \geq 1 , \quad Y_{n} | X_{n}  \sim \mathcal{B}\left( \pi \left( \theta^{T}X \right) \right) 
\] 
where $X_{1} , \ldots ,X_{n} , \ldots$ are independent and identically distributed random vectors lying in $\mathbb{R}^{d}$ and for all $x \in \mathbb{R}$, 
\[
\pi (x) = \frac{\exp (x)}{1+ \exp (x)} .
\]

\subsubsection{The WASNA} In the particular case of logistic regression, the weighted averaged version the stochastic Newton algorithm in Section \ref{subsec:ASN_algo} can be rewritten as:
\begin{equation}
\label{alg:WASNA_log}
\left\{
\begin{aligned}[2]
& \overline{a}_{n+1} = \pi \left( \theta_{n,\tau}^{T}X_{n+1} \right) \left( 1- \pi \left( \theta_{n,\tau}^{T}X_{n+1} \right) \right) \\
&  \tilde{\theta}_{n+1} = \tilde{\theta}_{n} + \gamma_{n+1} \overline{S}_{n}^{-1} X_{n+1} \left( Y_{n+1} - \pi \left( \tilde{\theta}_{n}^{T}X_{n+1} \right) \right) \\
& \theta_{n+1, \tau} = \left( 1- \tau_{n+1} \right) \theta_{n,\tau} + \tau_{n+1} \tilde{\theta}_{n+1} \\
& S_{n+1}^{-1} = S_{n}^{-1} - \left( 1+ a_{n+1}X_{n+1}^{T}S_{n}^{-1}X_{n+1} \right)^{-1} a_{n+1} S_{n}^{-1}X_{n+1}X_{n+1}^{T}S_{n}^{-1}
\end{aligned}
\right.
\end{equation}

with $\tilde{\theta}_{0} = \theta_{0 , \tau}$ bounded, $\overline{S}_{n}^{-1} = (n+1) S_{n}^{-1}$ and $S_{0}^{-1}$ symmetric and positive, $\gamma_{n} = c_{\gamma}n^{-\gamma}$ with $c_{\gamma} > 0$ and $\gamma \in (1/2,1)$, $\tau_{n} = \frac{\ln (n+1)^{\omega}}{\sum_{k=0}^{n}\ln(k+1)^{\omega}}$, with $\omega \geq 0$, 
\[
\overline{a}_{n+1} = \pi \left( X_{n+1}^{T} \theta_{n,\tau} \right) \left( 1-  \pi \left( X_{n+1}^{T} \theta_{n,\tau} \right) \right)
\]
and
\[ a_{n+1} = \max \left\lbrace \overline{a}_{n+1} , \frac{c_{\beta}}{(n+1)^{\beta}} \right\rbrace ,
\]
with $c_{\beta}> 0$ and $\beta \in (\gamma - 1/2)$.
Note that this choice of truncation gives Assumption  {\ref{(H1)}} for free, since 
$$
n^{\beta-1} S_n \succeq n^{\beta-1} \left( \lambda S_0 + \sum_{k=1}^n \frac{c_\beta}{k^\beta} X_k X_k^T \right) \xrightarrow[n\to \infty]{a.s.\ } \frac{c_\beta}{1-\beta}\mathbb{E} \left[ XX^T \right].
$$

Choosing $\gamma=1$, $c_\gamma=1$ and $\tau_n=1$ at each iteration leads to the Newton algorithm of Section \ref{sec:SN_algo} for the logistic regression, which matches the specific truncated Newton algorithm developed in \cite{BGBP2019}.

We study the convergence rates associated to this instance \eqref{alg:WASNA_log} of the WASNA in the following theorem.

\begin{theo}
\label{theo:ASNlog}
Suppose that $X$ admits a 4-th order moment and that 
$$\mathbb{E}\left[ \pi \left( \theta^{T}X \right) \left( 1- \pi \left( \theta^{T}X \right) \right) XX^{T} \right] \succ 0$$ 
is a positive matrix. The iterates given by the WASNA defined in \eqref{alg:WASNA_log} verify 
\[
\left\| \tilde{\theta}_{n} - \theta \right\|^{2 } = O \left( \frac{\ln n}{n^{\gamma}} \right) \quad a.s. \quad \quad \text{and} \quad \quad \left\| \theta_{n, \tau} - \theta \right\|^{2} = O \left( \frac{\ln n}{n} \right) \quad a.s.
\]
Furthermore,
\[
\sqrt{n} \left( \theta_{n, \tau} - \theta \right) \xrightarrow[n\to + \infty]{\mathcal{L}} \mathcal{N}\left( 0 , H^{-1} \right) .
\]
Finally,
\[
\left\| \overline{S}_{n} - H \right\|^{2} = O \left( \frac{1}{n^{2\beta}} \right) \quad a.s \quad \quad \text{and} \quad \quad \left\| \overline{S}_{n}^{-1} - H^{-1} \right\|^{2} = O \left( \frac{1}{n^{2\beta} } \right) \quad a.s.
\]
\end{theo}
The proof is given in Appendix \ref{app:proof_ASNlog}, and consists in verifying Assumptions \ref{(A1)}, \ref{(A2a)}, \ref{(A2b)}  {and \ref{(H1)} as well as the consistency of $\overline{S}_{n}$} to get the convergence rate of $\left\| \tilde{\theta}_{n} - \theta \right\|^{2 }$, then in verifying Assumptions \textbf{(A2c)} with Inequalities \eqref{majinthess} and \eqref{sumfini} to get the convergence rate of $\left\| \theta_{n, \tau} - \theta \right\|^{2}$ and the central limit theorem. Remark that in the context of the WASNA, we get the rates of convergence as for the direct truncated stochastic Newton algorithm  \cite{BGBP2019}, without additional assumptions.


\subsection{Softmax regression}\label{secsoft}

Let us recall that assuming that the number of different classes is $K$, and that the model parameters are  $\theta_{1} \in \mathbb{R}^{p} , \ldots , \theta_{K} \in \mathbb{R}^p$. Consider the samples $\left( X_{1} , Y_{1} \right) ,$ 
$\ldots , \left( X_{n} , Y_{n} \right), \ldots$ being i.i.d.\ random vectors in $\mathbb{R}^d \times [1,\hdots , K]$ with for all $n \geq 1$ and $k = 1 , \ldots , K$,
\[
\mathbb{P}\left[ Y_{n} = k | X_{n} \right] = \frac{e^{\theta_{k}^{T}X_{n}}}{\sum_{k' = 1}^{K} e^{\theta_{k'}^{T}X_{n}}}.
\]
Then, the likelihood can be written as
\[
L_{n}\left( \theta_{1} , \ldots , \theta_{K} \right) = \prod_{i=1}^{n} \frac{\sum_{k'=1}^{K}e^{\theta_{k'}^{T}X_{i}}\mathbf{1}_{Y_{i} = k}}{\sum_{k' = 1}^{K} e^{\theta_{k'}^{T}X_{n}}} = \prod_{i=1}^{n} \frac{e^{\theta_{Y_{i}}^{T}X_{i}}}{\sum_{k' = 1}^{K} e^{\theta_{k'}^{T}X_{i}}}  ,
\]
which leads to the following log-likelihood
\[
\ell_{n} \left( \theta_{1} , \ldots , \theta_{K} \right) = \sum_{i=1}^{n} \log \left( \frac{e^{\theta_{Y_{i}}^{T}X_{i}}}{\sum_{k' = 1}^{K} e^{\theta_{k'}^{T}X_{i}}}  \right) .
\]
Then, considering the asymptotic objective function, the aim is to minimize the convex function $G : \mathbb{R}^{d} \times \ldots \times \mathbb{R}^{d} \longrightarrow \mathbb{R}$ defined for all $h$ by
\[
G(h) = -\mathbb{E} \left[ \log \left( \frac{e^{h_{Y}^{T}X}}{\sum_{k' = 1}^{K} e^{h_{k'}^{T}X}}  \right) \right] =: \mathbb{E}\left[ g(X,Y,h) \right],
\]
where $\left( X,Y \right)$ is a i.i.d.\ copy of  $\left( X_{1} , Y_{1} \right) $. In order to establish convergence rate for the weighted averaged Newton algorithm.
If $X$ admits a second order moment, the functional $G$ is twice differentiable and for all $h = \left( h_{1} , \ldots , h_{K}\right) \in \mathbb{R}^{d} \times \ldots \times \mathbb{R}^{d}$,
\[
\nabla G (h) =  \mathbb{E}\left[ \left( 
X^{T} \left( \frac{e^{h_{1}^{T}X}}{\sum_{k= 1 }^{K} e^{h_{k}^{T}X}} - \mathbf{1}_{Y = 1} \right)  , \ldots ,
, X^{T} \left( \frac{e^{h_{K}^{T}X}}{\sum_{k= 1 }^{K} e^{h_{k}^{T}X}} - \mathbf{1}_{Y = K} \right) \right)^{T}
 \right] = \mathbf{E}\left[\nabla g \left( X , h \right) \right] .
\]
and one can check that $\nabla G (\theta)= 0$.
Furthermore, computing second-order derivatives leads to the Hessian, defined for all $h$ by
\[
\nabla^{2}G(h) = \mathbb{E}\left[\left( \text{diag}  \left( \sigma(X,h) \right) - \sigma(X,h)\sigma(X,h)^{T} \right) \otimes XX^{T} \right]
\]
where $
\sigma (X,h) = \left( 
 \frac{e^{h_{1}^{T}X}}{\sum_{k= 1 }^{K} e^{h_{k}^{T}X}} ,\ldots ,
\frac{e^{h_{K}^{T}X}}{\sum_{k= 1 }^{K} e^{h_{k}^{T}X}}
\right)^{T}$, and $\otimes$ denotes the Kronecker product. In what follows, that suppose that $\nabla^{2}G(\theta)$ is positive.

\subsubsection{The SNA}
In the case of softmax regression, the stochastic Newton algorithm can be defined for all $n \geq 1$ by
\begin{equation}
\label{alg:SNA_soft}
\left\{
\begin{aligned}[1]
& \Phi_{n+1} = \nabla_{h} g \left( X_{n+1},Y_{n+1} , \theta_{n} \right) \\
& \theta_{n+1} = \theta_{n} - \frac{1}{n+1} \overline{H}_{n}^{-1} \nabla_{h} g \left( X_{n+1} , Y_{n+1} , \theta_{n} \right) \\ 
& H_{n+\frac{1}{2}}^{-1} = H_{n}^{-1} - \left( 1+ \beta_{n+1}Z_{n+1}^{T} H_{n}^{-1} Z_{n+1} \right)^{-1} \beta_{n+1}H_{n}^{-1} Z_{n+1}Z_{n+1}^{T} H_{n}^{-1} \\
& H_{n+1} = H_{n+\frac{1}{2}}^{-1} - \left( 1+ \Phi_{n+1}^{T}H_{n+\frac{1}{2}}^{-1}\Phi_{n+1} \right)^{-1} H_{n+\frac{1}{2}}^{-1} \Phi_{n+1}\Phi_{n+1}^{T}H_{n+\frac{1}{2}}^{-1},
\end{aligned}
\right.
\end{equation}
where $\theta_{0}$ is bounded, $\overline{H}_{n} = (n+1)H_{n}^{-1}$, $H_{0}^{-1}$ is symmetric and positive, $\beta_{n} = c_{\beta}n^{-\beta}$, with $c_{\beta}> 0$ and $\beta\in (0,1/2)$, and $Z_{1} , \ldots , Z_{n+1}$ are i.i.d.\ with $Z_{1} \sim \mathcal{N}\left( 0 , I_{d} \right)$. 
To our knowledge, this is the first time that a stochastic Newton algorithm is made explicit for the softmax regression problem. 
The associated convergence guarantees follow.
\begin{theo}
\label{theo:ASNsoft}
 {Assume that $X$ admits a moment of order $4$ and that the Hessian $\nabla^{2} G(\theta)$ is positive,}
the iterates of the stochastic Newton algorithm defined by \eqref{alg:SNA_soft} satisfy 
\[
\left\| \theta_{n} - \theta \right\|^2 = O \left( \frac{\ln n}{n} \right) \quad a.s
\]
and
\[
\sqrt{n} \left( \theta_{n} - \theta \right)  \xrightarrow[n\to + \infty]{\mathcal{L}} \mathcal{N}\left( 0 , H^{-1} \right) .
\]
Furthermore,
\[
\left\| \overline{H}_{n} - H \right\|^{2} = O \left( \frac{1}{n^{2\beta}} \right) \quad a.s. \quad \quad \text{and} \quad \quad \left\| \overline{H}_{n}^{-1} - H^{-1} \right\|^{2} = O \left( \frac{1}{n^{2\beta}} \right) \quad a.s.
\]
\end{theo}
The proof is given in Appendix \ref{proof:ASNsoft}.

As far as we know, this is the first theoretical result covering the softmax regression using a stochastic Newton algorithm. In such a setting, the SNA proposes efficient online estimates and the convergence guarantees can be established  with weak assumptions {on the first moments of the features and about local strong convexity at the optimum only}.

\subsubsection{The WASNA}
Let us now consider the weighted averaged stochastic Newton algorithm defined recursively for all $n \geq 1$ by
\begin{equation}
\label{alg:WASNA_soft}
\left\{
\begin{aligned}[2]
& \overline{\Phi}_{n+1} = \nabla_{h} g \left( X_{n+1} , Y_{n+1} , {\theta}_{n,\tau} \right) \\
& \tilde{\theta}_{n+1} = \tilde{\theta}_{n} - \gamma_{n+1} \overline{S}_{n}^{-1} \nabla_{h} g \left( X_{n+1} , Y_{n+1} , \theta_{n} \right) \\
& \theta_{n+1, \tau} = \left( 1- \tau_{n+1} \right) \theta_{n, \tau} + \tau_{n+1} \tilde{\theta}_{n+1} \\
& S_{n+ \frac{1}{2}}^{-1} = S_{n}^{-1} - \left( 1+ \beta_{n+1} Z_{n+1}^{T}S_{n}^{-1}Z_{n+1} \right)^{-1} \beta_{n+1}S_{n}^{-1} Z_{n+1}Z_{n+1}^{T}S_{n}^{-1} \\
& S_{n+1}^{-1} = S_{n + \frac{1}{n}}^{-1} - \left( 1+ \overline{\Phi}_{n+1}^{T}S_{n + \frac{1}{n}}^{-1} \overline{\Phi}_{n+1} \right)^{-1} S_{n + \frac{1}{n}}^{-1} \overline{\Phi}_{n+1}\overline{\Phi}_{n+1}^{T}S_{n + \frac{1}{n}}^{-1} ,
\end{aligned}
\right.
\end{equation}
with $\tilde{\theta}_{0} = \theta_{0,\tau}$ bounded, $\overline{S}_{n}^{-1} = (n+1)S_{n}^{-1}$ and $S_{0}^{-1}$ symmetric and positive, $\gamma_{n} = c_{\gamma}n^{-\gamma}$ with $c_{\gamma}> 0$ and $\gamma \in (1/2,1)$, $\tau_{n} = \frac{\ln (n+1)^{\omega}}{\sum_{k=1}^{n}\ln(k+1)^{\omega}}$, with $\omega \geq 0$,  $\beta_{n} = c_{\beta}n^{-\beta}$ with $c_{\beta}>0$ and $\beta \in (\gamma - 1/2)$. Finally, $Z_{n}$ are i.i.d random vectors with $Z_{1} \sim \mathcal{N}\left( 0 ,I_{d} \right)$.
The following theorem gives rates of convergence of the weighted averaged stochastic Newton algorithm for the softmax regression.
\begin{theo}\label{ASNsoft}
 {Suppose that $X$ admits a moment of order $4$ and that the Hessian $\nabla^{2}G(\theta)$ is positive.} Then, the iterates of the weighted averaged stochastic Newton algorithm defined by \eqref{alg:WASNA_soft} satisfy 
\[
\left\| \tilde{\theta}_{n} - \theta \right\|^{2} = O \left( \frac{\ln n}{n^{\gamma}} \right) \quad a.s. \quad \quad \text{and} \quad \quad \left\| \theta_{n,\tau} - \theta \right\|^{2} = O \left( \frac{\ln n}{n} \right) \quad a.s,
\]
Furthermore,
\[
\sqrt{n} \left( \theta_{n, \tau} - \theta \right) \xrightarrow[n\to + \infty]{\mathcal{L}} \mathcal{N}\left( 0 , H^{-1} \right) .
\]
Finally,
\[
\left\| \overline{S}_{n} - S \right\|^{2} = O \left( \frac{1}{n^{2\beta}} \right) \quad a.s. \quad \quad \text{and} \quad \quad \left\| \overline{S}_{n}^{-1} - S^{-1} \right\|^{2} = O \left( \frac{1}{n^{2\beta}} \right) \quad a.s.
\]
\end{theo}
The proof is given in Appendix \ref{proof:ASNsoft}.
{This is the first time to our knowledge that a stochastic Newton algorithm is considered in the softmax regression setting, for which convergence results hold under weak assumptions. Note as well that the Riccatti's trick to update the inverse of the Hessian estimates is particularly appropriate, as the dimensionality of $\theta$ gets larger.} 


\subsubsection{Extra numerical experiments for the softmax regression}
\label{app:softmax_simu}

The considered multinomial regression model  is defined in the case of three-label classification in dimension $d=3$, for all $k=1,2,3,$ by
\[
\mathbb{P}\left[ Y=k |X \right] = \frac{e^{\theta_{k}^{T}X}}{\sum_{k' = 1}^{3}e^{\theta_{k'}^{T}X}}
\]
with $\theta = \left( \theta_{1}^{T} , \theta_{2}^{T} , \theta_{3}^{T} \right)^{T}$ chosen randomly on the unit sphereof $\mathbb{R}^{9}$. 
In Figure \ref{fig:softmax2}, we display the evolution of the quadratic mean error of the different estimates, for three different initializations, for correlated Gaussian variables $X \sim \mathcal{N}\left(0 , A \text{diag}\left( \frac{i^{2}}{d^{2}} \right)_{i=1 , \ldots , d} A^{T} \right)$ where $A$ is an orthogonal matrix randomly generated.
Results in the case of heteroscedastic Gaussian variables $X \sim \mathcal{N}\left(0 ,  \text{diag}\left( \frac{i^{2}}{d^{2}} \right)_{i=1 , \ldots , d}  \right)$ are very similar and can be found in Figure \ref{fig:softmax1}.
In Figure \ref{fig:softmax2}, one can see that again the averaged versions (weighted or not) converge faster in the case of bad initial point. The improvement over the Adagrad algorithm is made clearer as the initial point is chosen further from the optimum.


\begin{figure}[H]
\centering\includegraphics[width=0.97\textwidth]{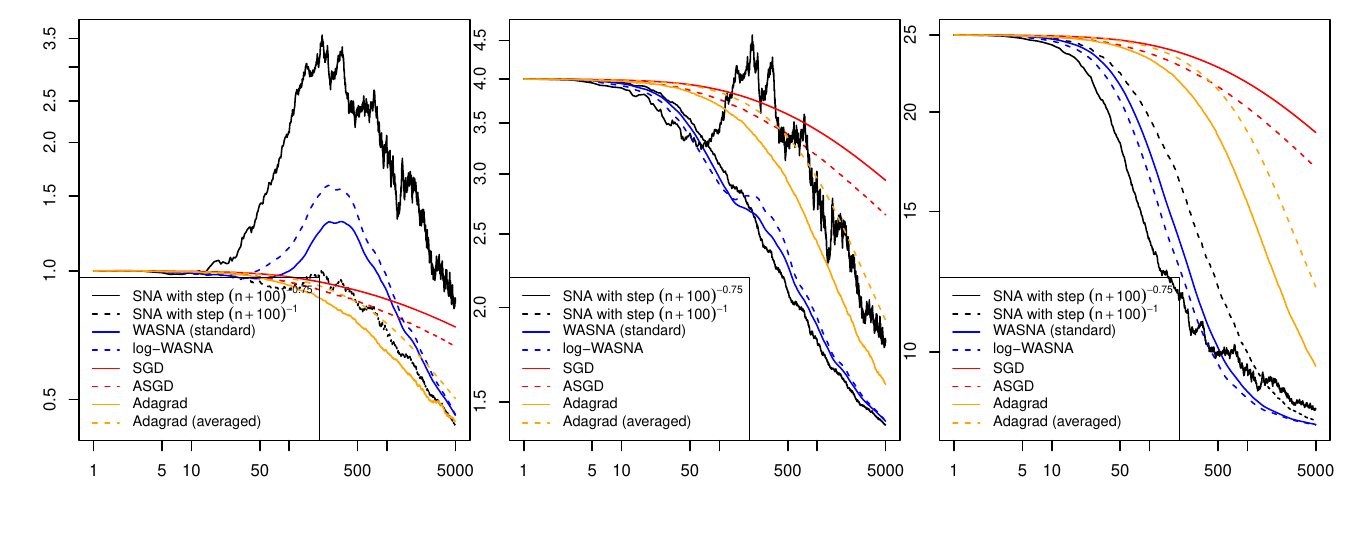}
\caption{\label{fig:softmax2}(Softmax regression with correlated Gaussian variables) Mean-squared error of the distance to the optimum $\theta$ with respect to the sample size for different initializations: $\theta_{0} = \theta + r_{0}U$, where $U$ is a uniform random variable on the unit sphere of $\mathbb{R}^{d}$, with  $r_{0} = 1$ (left), $r_0=2$ (middle) or $r_0=5$ (right). Each curve is obtained by an average over $50$ different samples of size $n=5000$ (drawing a different initial point each time).}
\end{figure}


\begin{figure}[H]
\centering\includegraphics[width=\textwidth]{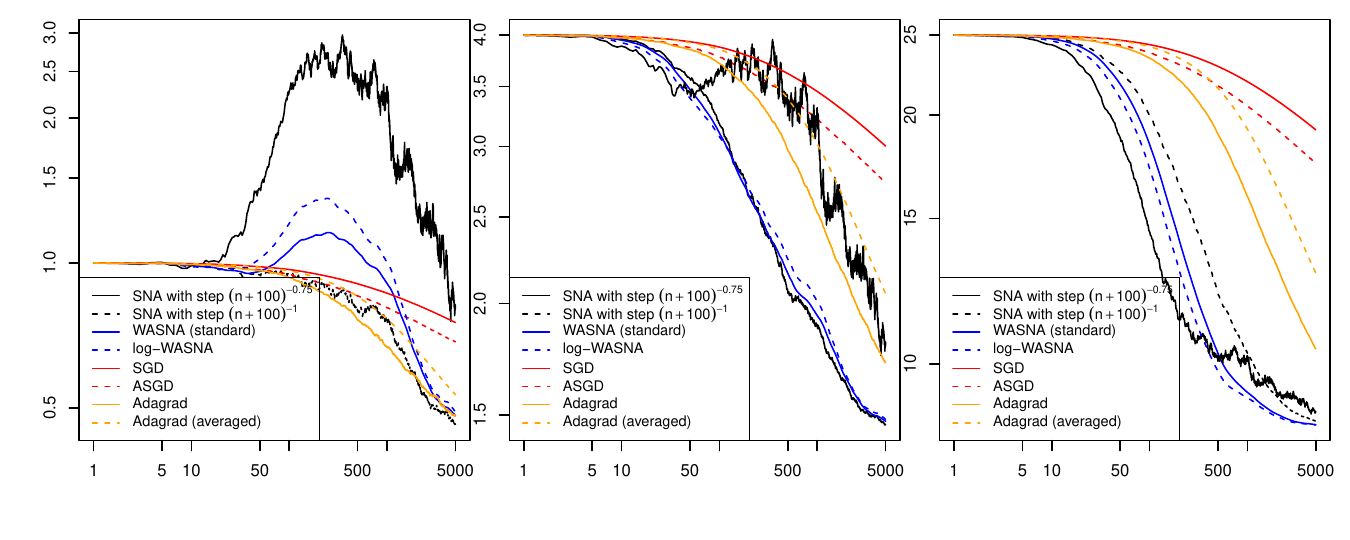}
\caption{\label{fig:softmax1}(Softmax regression with heteroscedastic Gaussian variables) Mean-squared error of the distance to the optimum $\theta$ with respect to the sample size for different initializations: $\theta_{0} = \theta + r_{0}U$, where $U$ is a uniform random variable on the unit sphere of $\mathbb{R}^{d}$, with  $r_{0} = 1$ (left), $r_0=2$ (middle) or $r_0=5$ (right). Each curve is obtained by an average over $50$ different samples of size $n=5000$ (drawing a different initial point each time).}
\end{figure}

\section{Proofs}\label{secproofs}

Remark that for the sake of simplicity, in all the proofs we consider that $c_{\theta}=c_{\gamma}'  = 0$. Indeed, for the cases where $c_{\theta}\neq 0$ or $c_{\gamma}' \neq 0$, one can consider $\overline{H'}_{n}^{-1} = \frac{n+1}{n+1+c_{\theta}} \overline{H}_{n}^{-1}$ and $\overline{S'}_{n}^{-1} = \frac{(n+1)^{\gamma}}{\left((n+1+c_{\gamma}\right)^{\gamma}}$. Then, these estimates have the same asymptotic behaviors as $\overline{H}_{n}^{-1}$ and $\overline{S}_{n}^{-1}$.
\subsection{Proof of Theorems \ref{snaas} and \ref{theothetanbar}}
\label{app:proof_as}
We only give the proof of Theorem \ref{theothetanbar} since taking $c_{\gamma} = 1 $ and $\gamma =1$ and exchanging $\overline{S}_{n}$ with $\overline{H}_{n}$ lead to the proof of Theorem \ref{snaas}.

With the help of a Taylor's decomposition of $G$, and thanks to Assumption \ref{(A2a)},
\begin{align*}
G \left( \tilde{\theta}_{n+1} \right) & = G \left( \tilde{\theta}_{n} \right) + \nabla G \left( \tilde{\theta}_{n} \right)^{T} \left( \tilde{\theta}_{n+1} - \tilde{\theta}_{n} \right) \\
&+ \frac{1}{2}\left( \tilde{\theta}_{n+1} - \tilde{\theta}_{n} \right)^{T}\int_{0}^{1} \nabla^{2}G \left( \tilde{\theta}_{n+1} + t \left( \tilde{\theta}_{n} - \tilde{\theta}_{n+1} \right) \right) dt \left( \tilde{\theta}_{n+1} - \tilde{\theta}_{n} \right) \\
& \leq G \left( \tilde{\theta}_{n} \right) + \nabla G \left( \tilde{\theta}_{n} \right)^{T} \left( \tilde{\theta}_{n+1} - \tilde{\theta}_{n} \right) + \frac{L_{\nabla G}}{2}\left\| \tilde{\theta}_{n+1} - \tilde{\theta}_{n} \right\|^{2}
\end{align*}
Then, since $\tilde{\theta}_{n+1} - \tilde{\theta}_{n} = - \gamma_{n+1} \overline{S}_{n}^{-1}\nabla_{h} g \left( X_{n+1} , \tilde{\theta}_{n} \right)$,
\begin{align*}
G \left( \tilde{\theta}_{n+1} \right) & = G \left( \tilde{\theta}_{n} \right) - \gamma_{n+1} \nabla G \left( \tilde{\theta}_{n} \right)^{T} \overline{S}_{n}^{-1} \nabla_{h} g \left( X_{n+1} , \tilde{\theta}_{n} \right)  \\
&\quad + \frac{L_{\nabla G}}{2}\gamma_{n+1}^{2}\left\| \overline{S}_{n}^{-1} \nabla_{h} g \left( X_{n+1} , \tilde{\theta}_{n} \right) \right\|^{2} \\
& \leq G \left( \tilde{\theta}_{n} \right) - \gamma_{n+1} \nabla G \left( \tilde{\theta}_{n} \right)^{T} \overline{S}_{n}^{-1} \nabla_{h} g \left( X_{n+1} , \tilde{\theta}_{n} \right) \\
&\quad + \frac{L_{\nabla G}}{2} \gamma_{n+1}^{2} \left\| \overline{S}_{n}^{-1} \right\|_{op}^{2}\left\|  \nabla_{h} g \left( X_{n+1} , \tilde{\theta}_{n} \right) \right\|^{2}
\end{align*}
Let us define $V_{n} = G \left( \tilde{\theta}_{n} \right) - G (\theta ) $. Then, we can rewrite the previous inequality as
\begin{align*}
V_{n+1} \leq V_{n} - \gamma_{n+1} \nabla G \left( \tilde{\theta}_{n} \right)^{T} \overline{S}_{n}^{-1} \nabla_{h} g \left( X_{n+1} , \tilde{\theta}_{n} \right)  + \frac{L_{\nabla G}}{2} \gamma_{n+1}^{2}\left\| \overline{S}_{n}^{-1} \right\|_{op}^{2}\left\|  \nabla_{h} g \left( X_{n+1} , \tilde{\theta}_{n} \right) \right\|^{2}
\end{align*}
and considering the conditional expectation w.r.t.\ $\mathcal{F}_{n}$, since $\tilde{\theta}_{n}$ and $\overline{S}_{n}^{-1}$ are $\mathcal{F}_{n}$-measurable,
\begin{align*}
\mathbb{E}\left[ V_{n+1} |\mathcal{F}_{n} \right] &\leq V_{n} - \gamma_{n+1} \nabla G \left( \tilde{\theta}_{n} \right)^{T} \overline{S}_{n}^{-1}\nabla G  \left(  \tilde{\theta}_{n} \right)  \\
&\quad + \frac{L_{\nabla G}}{2} \gamma_{n+1}^{2}\left\| \overline{S}_{n}^{-1} \right\|_{op}^{2} \mathbb{E}\left[ \left\|  \nabla_{h} g \left( X_{n+1} , \tilde{\theta}_{n} \right) \right\|^{2} |\mathcal{F}_{n} \right] .
\end{align*}
Then, thanks to Assumption \ref{(A1b)}, it comes
\begin{align*}
\mathbb{E}\left[ V_{n+1} |\mathcal{F}_{n} \right] &\leq  \left( 1 + \frac{C'L_{\nabla G}}{2}\gamma_{n+1}^{2} \left\| \overline{S}_{n}^{-1} \right\|_{op}^{2} \right) V_{n} - \gamma_{n+1} \lambda_{\min} \left( \overline{S}_{n}^{-1} \right) \left\| \nabla G \left( \tilde{\theta}_{n} \right) \right\|^{2} \\
&\quad + \frac{C L_{\nabla G}}{2} \gamma_{n+1}^{2} \left\| \overline{S}_{n}^{-1} \right\|_{op}^{2}
\end{align*}
Remark that thanks to Assumption  {\ref{(H1)}}, 
\[
\sum_{n\geq 0}\gamma_{n+1}^{2} \left\| \overline{S}_{n}^{-1} \right\|_{op}^{2} < + \infty \quad a.s .
\]
Then, since $\overline{S}_{n}^{-1}$ is positive and applying Robbins-Siegmund Theorem, $V_{n}$ almost surely converges to a finite random variable and
\[
\sum_{n \geq 0}\gamma_{n+1} \lambda_{\min} \left( \overline{S}_{n}^{-1} \right) \left\| \nabla G \left( \tilde{\theta}_{n} \right) \right\|^{2} < + \infty \quad p.s
\]
and since, thanks to Assumption  {\ref{(H1)}}, 
\[
\sum_{n\geq 0} \gamma_{n+1} \lambda_{\min} \left( \overline{S}_{n}^{-1} \right) = \sum_{n\geq 0} \gamma_{n+1} \frac{1}{\lambda_{\max} \left( \overline{S}_{n} \right)} = + \infty \quad a.s,
\]
this necessarily implies that $\liminf_{n} \left\| \nabla G \left( \tilde{\theta}_{n} \right) \right\| = 0$ almost surely. Since $G$ is strictly convex, this also implies that
\[
\liminf_{n} \left\| \tilde{\theta}_{n} - \theta \right\| = 0 \quad a.s \quad \quad \text{and} \quad \quad \liminf_{n} V_{n} = G\left( \tilde{\theta}_{n} \right) - G (\theta ) = 0 \quad a.s,
\]
and since $V_{n}$ converges almost surely to a random variable, this implies that $G\left( \tilde{\theta}_{n} \right)$ converges almost surely to $G(\theta)$. Finally, by strict convexity, $\tilde{\theta}_{n}$ converges almost surely to $\theta$.

\subsection{Proof of Theorem \ref{theoratetheta}}
\label{app:proof_rate_SN}
 Following the strategy used in \cite[Theorem 4.2]{BGBP2019}, the main ideas of the proof are the following. Remark that 
\begin{align}
\label{decxitheta} \theta_{n+1} - \theta & = \theta_{n} - \theta - \frac{1}{n+1}\overline{H}_{n}^{-1} \nabla G \left( \theta_{n} \right) + \frac{1}{n+1}\overline{H}_{n}^{-1}\xi_{n+1} \\
\notag & = \theta_{n} - \theta - \frac{1}{n+1}H^{-1} \nabla G \left( \theta_{n} \right) - \frac{1}{n+1}\left( \overline{H}_{n}^{-1} - H^{-1} \right)\nabla G\left( \theta_{n} \right) \\
\notag & + \frac{1}{n+1}\overline{H}_{n}^{-1}\xi_{n+1}  \\
\label{decdeltatheta} & = \left( 1- \frac{1}{n+1}\right)\left( \theta_{n} - \theta \right) - \frac{1}{n+1}H^{-1}\delta_{n} \\
\notag & - \frac{1}{n+1}\left( \overline{H}_{n}^{-1}- H^{-1} \right)\nabla G \left( \theta_{n} \right) + \frac{1}{n+1}\overline{H}_{n}^{-1} \xi_{n+1}
\end{align}
with $\xi_{n+1} = \nabla G \left( \theta_{n} \right) - \nabla_{h} g \left( X_{n+1} , \theta_{n} \right) $ and $\delta_{n} = \nabla G \left( \theta_{n} \right) - H \left( \theta_{n} - \theta \right)$ is the remainder term in the Taylor's decomposition of the gradient. Since $\theta_{n}$ and $\overline{H}_{n}$ are $\mathcal{F}_{n}$-measurable, and since $X_{n+1}$ is independent from $\mathcal{F}_{n}$, $\left( \xi_{n+1} \right)$ is a sequence of martingale differences adapted to the filtration $\left( \mathcal{F}_{n} \right)$. Moreover, inductively, one can check that
\begin{equation}
\label{dectheta} \theta_{n+1} - \theta = \frac{1}{n+1}\underbrace{\sum_{k=0}^{n} \overline{H}_{k}^{-1}\xi_{k+1}}_{=:M_{n+1}} \underbrace{- \frac{1}{n+1}\sum_{k=0}^{n}H^{-1}\delta_{k} - \frac{1}{n+1}\sum_{k=0}^{n} \left( \overline{H}_{k}^{-1} - H^{-1} \right) \nabla G \left( \theta_{k} \right)}_{=:\Delta_n} .
\end{equation}

\subsubsection{Convergence rate for $M_{n+1}$}
Note that $\left( M_{n} \right)$ is a martingale adapted to the filtration $\left( \mathcal{F}_{n} \right)$ and since $\overline{H}_{n}^{-1}$ is $\mathcal{F}_{n}$-measurable,
\begin{align*}
\left\langle M \right\rangle_{n+1} &  = \sum_{k=0}^{n} \overline{H}_{k}^{-1}\mathbb{E}\left[ \xi_{k+1}\xi_{k+1}^{T} |\mathcal{F}_{k} \right] \overline{H}_{k}^{-1} \\
& = \sum_{k=0}^{n} \overline{H}_{k}^{-1} \mathbb{E}\left[ \nabla_{h} g \left( X_{k+1} , \theta \right) \nabla_{h} g \left( X_{k+1} , \theta_{k} \right)^{T} |\mathcal{F}_{k} \right] \overline{H}_{k}^{-1} \\
&- \sum_{k=0}^{n} \overline{H}_{k}^{-1} \nabla G \left( \theta_{k} \right) \nabla G \left( \theta_{k} \right)^{T} \overline{H}_{k}^{-1}
\end{align*}  
Since $\overline{H}_{k}^{-1}$ converges almost surely to $H^{-1}$, since $\theta_{k}$ converges almost surely to $\theta$, and since $\nabla G$ is $L_{\nabla G}$ lipschitz,
\[
\frac{1}{n+1}\sum_{k=0}^{n} \overline{H}_{k}^{-1} \nabla G \left( \theta_{k} \right) \nabla G \left( \theta_{k} \right)^{T} \overline{H}_{k}^{-1} \xrightarrow[n\to + \infty]{a.s} 0.
\]
Moreover, Assumption \ref{(A1c)} entails that
\begin{align*}
\frac{1}{n+1}\sum_{k=0}^{n} \overline{H}_{k}^{-1} \mathbb{E}\left[ \nabla_{h} g \left( X_{k+1} , \theta \right) \nabla_{h} g \left( X_{k+1} , \theta_{k} \right)^{T} |\mathcal{F}_{k} \right] & \overline{H}_{k}^{-1} \\
\xrightarrow[n\to + \infty]{a.s}  H^{-1} &\mathbb{E}\left[ \nabla_{h} g \left( X , \theta \right)\nabla_{h} g (X,\theta)^{T} \right] H^{-1} .
\end{align*}
Setting $\Sigma := \mathbb{E}\left[ \nabla_{h} g \left( X , \theta \right)\nabla_{h} g (X,\theta)^{T} \right]$, one has
\begin{equation}
\label{convmn} \frac{1}{n+1}\left\langle M \right\rangle_{n+1} \xrightarrow[n\to + \infty]{a.s} H^{-1}\Sigma H^{-1} .
\end{equation}
Then, applying a law of large numbers for martingales (see \cite{Duf97}), for all $\delta > 0$,
\[
\frac{1}{n^{2}}\left\| M_{n+1} \right\|^{2} = o \left( \frac{(\ln n)^{1+\delta}}{n} \right) \quad a.s,
\]
and if inequality \eqref{momenta} is verified,
\[
\frac{1}{n^{2}}\left\| M_{n+1} \right\|^{2} = O \left( \frac{\ln n}{n} \right) \quad a.s.
\]

\subsubsection{Convergence rate for $\Delta_{n}$} Let us recall that
\[
\Delta_{n} :=- \frac{1}{n+1}\sum_{k=1}^{n}H^{-1}\delta_{k} - \frac{1}{n+1}\sum_{k=1}^{n} \left( \overline{H}_{k}^{-1} - H^{-1} \right) \nabla G \left( \theta_{k} \right).
\]
Given that $\theta_{k} $ converges almost surely to $\theta$, and since the Hessian is continuous at $\theta$, 
\[
\left\| \delta_{n} \right\| = \left\| \int_{0}^{1} \left(  \nabla^{2} G \left( \theta + t\left( \theta_{n} - \theta \right) - \nabla^{2}G(\theta) \right) \right) dt \left( \theta_{n} - \theta \right) \right\| = o \left( \left\| \theta_{n} - \theta \right\| \right) \quad a.s. 
\]
Similarly, since the gradient is $L_{\nabla G}$-lipschitz, one can check that 
\[
\left\| \left( \overline{H}_{k}^{-1} - H^{-1} \right) \nabla G \left( \theta_{n} \right) \right\| = o \left( \left\| \theta_{n} - \theta \right\| \right) \quad a.s 
\]
Then, on has
\begin{align*}
    \left\| \Delta_{n+1} \right\|  \leq  \frac{n}{n+1}\left\| \Delta_{n}\right\| + \frac{1}{n+1}\left\| \delta_{n} \right\| & \leq \frac{n}{n+1}\left\| \Delta_{n} \right\| + \frac{1}{n+1} o \left( \left\| \theta_{n} - \theta \right\| \right) \quad a.s \\
    & = \frac{n}{n+1}\left\| \Delta_{n} \right\|  + \frac{1}{n+1}\left( \left\|\Delta_{n} \right\| + \left\|M_{n} \right\| \right) \quad a.s \\
    & = \left( 1- \frac{1}{n+1} + o \left( \frac{1}{n+1}\right) \right) \left\| \Delta_{n} \right\| + \frac{1}{n+1} o \left( \left\|M_{n} \right\| \right) \quad a.s
\end{align*}
and following computation such as Equations (6.27),(6.28),(6.34) and (6.35) in \cite{BGBP2019}, one has
\[
\left\| \Delta_{n} \right\|^{2} = O \left( \frac{1}{n^{2}} \left\| M_{n+1} \right\|^{2} \right) \quad a.s,
\]
which concludes the proof.

\subsection{Proof of Theorem \ref{thetatlc}}
\label{app:proof_tlc_SN}
In order to derive asymptotic normality of the estimates, start again from Equation \eqref{dectheta}: the first term will dictate the speed of convergence, while the other terms will collapse. First, since $\left( \xi_{n} \right)$ is a sequence of martingale differences adapted to the filtration $\left( \mathcal{F}_{n} \right)$, given \eqref{convmn} and \eqref{momenta}, the Central Limit Theorem for martingales (see \cite{Duf97}) reads as follows,
\begin{align}
\label{eq:TCL_Mn}
\frac{1}{\sqrt{n}} \sum_{k=0}^{n} \overline{H}_{k}^{-1} \xi_{k+1} &\xrightarrow[n\to + \infty]{\mathcal{L}} \mathcal{N} \left( 0 , H^{-1} \Sigma H^{-1} \right) .
\end{align}
It remains to show that the other terms on the right-hand side of  \eqref{dectheta} are negligible. 
Under Assumption \ref{(A2c)}, and since $\theta_{n}$ converges almost surely to $\theta$,
\begin{align*}
\left\| \delta_{n} \right\| &  = \left\| \int_{0}^{1} \left(  \nabla^{2} G \left( \theta + t\left( \theta_{n} - \theta \right) - \nabla^{2}G(\theta) \right) \right) dt \left( \theta_{n} - \theta \right) \right\| = O \left( \left\| \theta_{n} - \theta \right\|^{2} \right) \quad a.s.
\end{align*}
Theorem \ref{theoratetheta} coupled with the Toeplitz lemma imply in turn 
\begin{align}
\label{eq:speed_delta_n}
\frac{1}{n+1}\left\| \sum_{k=0}^{n} \delta_{k} \right\| &= o \left( \frac{(\ln n)^{2+\delta}}{n} \right) \quad a.s
\end{align}
In the same way,  the gradient being $L_{\nabla G}$-lipschitz and under Assumption \ref{(H2b)}, one has 
\begin{align}
\label{eq:speed_third_term}
\frac{1}{n+1}\left\| \sum_{k=0}^{n} \left( \overline{H}_{k}^{-1} - H^{-1}\right) \nabla G \left( \theta_{k} \right) \right\| = o \left(  \frac{(\ln n)^{2+\delta}}{n^{\min \left\lbrace \frac{1}{2} + p_{H} , 1 \right\rbrace}} \right) \quad a.s.
\end{align}
The  rates obtained in \eqref{eq:speed_delta_n} and \eqref{eq:speed_third_term} are negligible compared to the one in \eqref{eq:TCL_Mn}, which leads to the desired conclusion.

\subsection{Proof of Theorem \ref{theothetatilde}}
\label{app:proof_theothetatilde}
Considering the Weighted Averaged Stochastic Newton Algorithm defined by \eqref{def::genmoy}, Inequality \eqref{decdeltatheta} can be adapted such as
\begin{align}
\label{decdeltatilde} \tilde{\theta}_{n+1} - \theta =& \left( 1- \gamma_{n+1} \right) \left( \tilde{\theta}_{n} - \theta \right) - \gamma_{n+1}H^{-1} \tilde{\delta}_{n} \\
\notag & - \gamma_{n+1}\left( \overline{S}_{n}^{-1} - H^{-1} \right) \nabla G \left( \tilde{\theta}_{n} \right) + \gamma_{n+1}\overline{S}_{n}^{-1}\tilde{\xi}_{n+1} 
\end{align}
with $\tilde{\xi}_{n+1} = \nabla G \left( \tilde{\theta}_{n} \right) - \nabla_{h} g \left( X_{n+1} , \tilde{\theta}_{n} \right) $ and $\tilde{\delta}_{n} = \nabla G \left( \tilde{\theta}_{n} \right) - H \left( \tilde{\theta}_{n} - \theta \right)$ is the remainder term of the Taylor's expansion of the gradient. Since $\tilde{\theta}_{n},S_{n}$ are $\mathcal{F}_{n}$-measurable and $X_{n+1}$ is independent from $\mathcal{F}_{n}$, $\left( \tilde{\xi}_{n+1} \right)$ is  a sequence of martingale differences adapted to the filtration $\left( \mathcal{F}_{n} \right)$. Moreover, inductively, one can check that

\begin{align}
 \label{decbeta} \tilde{\theta}_{n} - \theta  = & \beta_{n,0}\left(  \tilde{\theta}_{0} - \theta \right) \\
 \notag & \underbrace{- \sum_{k=0}^{n-1} \beta_{n,k+1}\gamma_{k+1}H^{-1}\tilde{\delta}_{k} - \sum_{k=0}^{n-1} \beta_{n,k+1}\gamma_{k+1} \left( \overline{S}_{k}^{-1} -H^{-1} \right) \nabla G \left( \tilde{\theta}_{k} \right)}_{:= \tilde{\Delta}_{n}}  \\
 \notag & + \underbrace{\sum_{k=0}^{n-1} \beta_{n,k+1} \gamma_{k+1}\overline{S}_{k}^{-1}\tilde{\xi}_{k+1} }_{:= \tilde{M}_{n}} 
\end{align}
with for all $k,n \geq 0$ such that $k \leq n$, $\beta_{n,k} = \prod_{j=k+1}^{n} \left( 1- \gamma_{j} \right)$ and $\beta_{n,n}=1$. Focusing on the last term of \eqref{decbeta},  applying Theorem 6.1 in \cite{CGBP2020} and thanks to Inequality \eqref{momenteta}
\begin{align}
\label{eq:speed_av_last_term}
\left\| \sum_{k=0}^{n-1} \beta_{n,k+1} \gamma_{k+1} \overline{S}_{k}^{-1} \tilde{\xi}_{k+1} \right\|^{2} = O \left( \frac{\ln n}{n^{\gamma}} \right) \quad a.s.
\end{align}
Furthermore, one can check that
\[
\left| \beta_{n,0} \right| = O \left( \exp \left( - \frac{c_{\gamma}}{1- \gamma}n^{1-\gamma} \right) \right) ,
\]
and the term $\beta_{n,0}\left( \tilde{\theta}_{0} - \theta \right)$ is negligible compared to \eqref{eq:speed_av_last_term}. 
Considering now 
$\tilde{\Delta}_{n} $ and following the proof of Theorem \ref{theoratetheta}, one can check that
\[
\left\| \tilde{\delta}_{n} \right\| = o \left( \left\| \tilde{\theta}_{n} - \theta \right\| \right) \quad a.s \quad \quad \text{and} \quad \quad \left\| \left( \overline{S}_{n}^{-1} - H^{-1 }\right) \nabla G \left( \tilde{\theta}_{n} \right) \right\| = o \left( \left\| \tilde{\theta}_{n} - \theta \right\| \right) \quad a.s
\]
Let $n_{0}$ such that for all $n \geq n_{0}$, $\gamma_{n} \leq 1$. Then, for all $n \geq n_{0}$, 
\begin{align*}
\left\| \tilde{\Delta}_{n+1} \right\|  \leq & \left( 1- \gamma_{n+1} \right) \left\| \tilde{\Delta}_{n} \right\| + \gamma_{n+1} \left( \left\| H^{-1} \tilde{\delta}_{n} \right\| + \left\| \left( \overline{S}_{n}^{-1} - H^{-1 }\right) \nabla G \left( \tilde{\theta}_{n} \right) \right\|  \right) \\
 = & \left( 1- \gamma_{n+1} \right) \left\| \tilde{\Delta}_{n} \right\| +  o \left(\gamma_{n+1} \left\| \tilde{\theta}_{n} - \theta \right\| \right) \quad a.s \\
 = & \left( 1- \gamma_{n+1} \right) \left\| \tilde{\Delta}_{n} \right\| + o \left( \gamma_{n+1} \left\| \tilde{M}_{n} + \beta_{n,0}\left( \tilde{\theta}_{0} - \theta \right) \right\| + \gamma_{n+1} \left\| \tilde{\Delta}_{n} \right\| \right) \quad a.s \\
 = & \left( 1- \gamma_{n+1} + o(\gamma_{n+1}) \right) \left\| \tilde{\Delta}_{n} \right\| \\
&+ o \left( \gamma_{n+1} \left( \sqrt{\frac{\ln n}{n^{\gamma}}} + \exp \left( - \frac{c_{\gamma}}{1-\gamma}n^{1-\gamma} \right) \right) \right) \quad a.s
\end{align*}
and applying a lemma of stabilization \citep{Duf97} or with analogous calculus to the ones of the proof of Lemma 3 in \cite{pelletier1998almost}, it comes
\[
\left\| \tilde{\Delta}_{n}\right\|^{2} = O \left( \frac{\ln n}{n^{\gamma}} \right) \quad a.s.
\]

\subsection{Proof of Theorem \ref{theomoy}}
\label{app:proof_thm_theomoy}

Remark that for all $n \geq 0$, $\theta_{n,\tau}$ can be written as
\begin{equation}
\label{decgenmoy} \theta_{n,\tau} - \theta =  \underbrace{ \prod_{k=j}^{n}\left( 1- \tau_{j} \right) }_{=: \kappa_{n,0}} \left( \theta_{0,\tau} - \theta \right) + \sum_{k=0}^{n-1}\underbrace{\prod_{j=k+1}^{n} \left( 1- \tau_{j} \right)}_{=: \kappa_{n,k}} \tau_{k+1} \left( \tilde{\theta}_{k} - \theta \right) .
\end{equation}
with $\kappa_{n,n} = 1$. Remark also that \eqref{decdeltatilde} can be written as
\begin{equation}
\tilde{\theta}_{n} - \theta = \frac{ \tilde{\theta}_{n} - \theta - \left( \tilde{\theta}_{n+1} - \theta \right)}{\gamma_{n+1}} - H^{-1} \tilde{\delta}_{n} -  \left( \overline{S}_{n}^{-1} - H^{-1} \right) \nabla G \left( \tilde{\theta}_{n} \right) + \overline{S}_{n}^{-1} \tilde{\xi}_{n+1}
\end{equation}
Then, \eqref{decgenmoy} can be written as
\begin{align}
\notag \theta_{n,\tau} - \theta = \kappa_{n,0} \left( \theta_{0,\tau} - \theta \right) + \sum_{k=1}^{n} \kappa_{n,k}\tau_{k+1}\frac{ \tilde{\theta}_{k} - \theta - \left( \tilde{\theta}_{k+1} - \theta \right)}{\gamma_{k+1}} - H^{-1} \sum_{k=0}^{n-1} \kappa_{n,k}\tau_{k+1}\tilde{\delta}_{k} \\
\label{decdeporc} -  \sum_{k=0}^{n-1} \kappa_{n,k}\tau_{k+1}\left( \overline{S}_{k}^{-1} - H^{-1} \right) \nabla G \left( \tilde{\theta}_{k} \right) + \sum_{k=0}^{n-1} \kappa_{n,k}\tau_{k+1}\overline{S}_{k}^{-1} \tilde{\xi}_{k+1}
\end{align}
The rest of the proof consists in establishing the rate of convergence of each term on the right-hand side  \eqref{decdeporc}.

\medskip

\noindent \textbf{Bounding $\left\| \prod_{k=1}^{n}\left( 1- \tau_{k} \right) \left( \theta_{0,\tau} - \theta \right) \right\|$:  } Since $n \tau_{n}$ converges to $\tau > 1/2$, there is a rank $n_{\tau}$ such that for all $n \geq n_{\tau}$, $0 \leq \tau_{n} \leq 1$, so that for all $n \geq n_{\tau}$,
\begin{align}
\notag
\prod_{k=1}^{n} \left| 1- \tau_{k} \right| &\leq \prod_{k=1}^{n_{\tau}-1} \left| 1- \tau_{k} \right| \exp \left( \sum_{k=n_{\tau}}^{n}  1- \tau_{k}  \right) \\
\label{majtaun}
&\leq \prod_{k=1}^{n_{\tau}-1} \left| 1- \tau_{k} \right| \exp \left( - \sum_{k=n_{\tau}}^{n} \tau_{k} \right) = O \left( \frac{1}{n^{\tau}} \right) .
\end{align}
Then
\[
\left\| \prod_{k=1}^{n}\left( 1- \tau_{k} \right) \left( \theta_{0,\tau} - \theta \right) \right\| = O \left( \frac{1}{n^{\tau}} \right) \quad a.s.
\]
and this term is negligible since $\tau > 1/2$.
\medskip

\noindent \textbf{Bounding $\sum_{k=0}^{n-1} \kappa_{n,k}\tau_{k+1}\tilde{\delta}_{k}$:  } Since
\[
\left\| \tilde{\delta}_{n} \right\| = O \left( \left\| \tilde{\theta}_{n} - \theta \right\|^{2} \right) \quad a.s,
\]
and with the help of Theorem \ref{theothetatilde}, for all $\delta > 0$, there is a positive random variable $B_{\delta}$ such that
\[
\left\| \sum_{k=0}^{n-1} \kappa_{n,k}\tau_{k+1}\tilde{\delta}_{k} \right\| \leq B_{\delta} \sum_{k=1}^{n} \left| \kappa_{n,k} \right|\tau_{k+1}\frac{(\ln (+1))^{1+\delta}}{(k+1)^{\gamma}}  \quad a.s.
\]
Then, since the sequence $\left( \frac{(\ln n)^{1+\delta}}{n^{\gamma}} \right)$ is in $\mathcal{G}\mathcal{S}( - \gamma)$, applying Lemma \ref{lemmok}, 
\[
\left\| \sum_{k=0}^{n-1} \kappa_{n,k}\tau_{k+1}\tilde{\delta}_{k} \right\| = o \left( \frac{(\ln n)^{1+\delta}}{n^{\gamma}} \right) \quad a.s.
\]

\medskip

\noindent\textbf{Bounding $\sum_{k=0}^{n-1} \kappa_{n,k}\tau_{k+1}\left( \overline{S}_{k}^{-1} - H^{-1} \right) \nabla G \left( \tilde{\theta}_{k} \right)$: } Thanks to Assumption {\ref{(H2b)}} and since the gradient of $G$ is $L_{\nabla G}$ lipshitz, for all $\delta > 0$, there is a positive random variable $B_{\delta}'$ such that
\[
\left\| \sum_{k=0}^{n-1} \kappa_{n,k}\tau_{k}\left( \overline{S}_{k}^{-1} - H^{-1} \right) \nabla G \left( \tilde{\theta}_{k} \right) \right\| \leq B_{\delta}' \sum_{k=0}^{n-1} \left|\kappa_{n,k}\right| \tau_{k+1} \frac{(\ln (k+1))^{1/2 + \delta}}{(k+1)^{p_{H}+ \gamma/2}} \quad a.s.
\]
Then, since the sequence $\left( \frac{(\ln n)^{1/2 + \delta}}{n^{p_{H} + \gamma/2}} \right)$ is in $\mathcal{G}\mathcal{S}(  -p_{H} - \gamma /2)$, applying Lemma \ref{lemmok},
\[
\left\| \sum_{k=0}^{n-1} \kappa_{n,k}\tau_{k+1}\left( \overline{S}_{k}^{-1} - H^{-1} \right) \nabla G \left( \tilde{\theta}_{k} \right) \right\| = o \left( \frac{(\ln n)^{1/2 + \delta}}{n^{p_{H} + \gamma /2}} \right) \quad a.s.
\]

\medskip

\noindent\textbf{Bounding $\sum_{k=0}^{n-1} \kappa_{n,k}\tau_{k+1}\frac{ \tilde{\theta}_{k} - \theta - \left( \tilde{\theta}_{k+1} - \theta \right)}{\gamma_{k+1}}$: } Applying an Abel's transform, one can check that
\begin{align*}
\sum_{k=0}^{n-1} \kappa_{n,k}\tau_{k+1} &  \frac{ \tilde{\theta}_{k} - \theta - \left( \tilde{\theta}_{k+1} - \theta \right)}{\gamma_{k+1}} \\
=& \frac{\kappa_{n,0}\tau_{1}}{\gamma_{1}} \left( \tilde{\theta}_{0} - \theta \right) - \frac{\tau_{n}}{\gamma_{n}}\left( \tilde{\theta}_{n} - \theta \right) \\
&+ \sum_{k=1}^{n-1} \kappa_{n,k} \tau_{k+1} \gamma_{k+1}^{-1} \left( 1- \left(1- \tau_{k+1} \right)\frac{\tau_{k} \gamma_{k}^{-1}}{\tau_{k+1} \gamma_{k+1}^{-1}} \right) \left( \tilde{\theta}_{k} - \theta \right) 
\end{align*}
Thanks to \eqref{majtaun},
\[
\frac{\kappa_{n,1}\tau_{1}}{\gamma_{1}} \left\| \tilde{\theta}_{0} - \theta \right\| = O \left( \frac{1}{n^{\tau}} \right) \quad a.s .
\]
Furthermore, using Theorem \ref{theothetatilde}, 
\[
\frac{\tau_{n+1}}{\gamma_{n+1}}\left\| \tilde{\theta}_{n+1} - \theta \right\| = O \left( \frac{\sqrt{ \ln n}}{n^{1- \gamma /2}}  \right) \quad a.s.
\]
Finally, since $\tau_{n}$ is in $\mathcal{G}\mathcal{S}(-1 )$,
\begin{align*}
 1- \left(1- \tau_{n+1} \right)\frac{\tau_{n} \gamma_{n}^{-1}}{\tau_{n} \gamma_{n+1}^{-1}}  =&  1 + \underbrace{\left(1- \tau_{n+1} \right)}_{= 1 - \frac{\tau}{n} + o \left( \frac{1}{n}\right)}  \underbrace{\left( 1- \frac{\tau_{n}}{\tau_{n+1}} \right)}_{= \frac{-1}{n} + o \left(\frac{1}{n} \right)}\underbrace{\frac{ \gamma_{n}^{-1}}{ \gamma_{n+1}^{-1}}}_{=1+ \frac{\gamma}{n} + o \left( \frac{1}{n}\right)} \\
 &-   \underbrace{\left(1- \tau_{n+1} \right)}_{= 1 - \frac{\tau}{n} + o \left( \frac{1}{n}\right)} \underbrace{\frac{ \gamma_{n}^{-1}}{ \gamma_{n+1}^{-1}}}_{=1+ \frac{\gamma}{n} + o \left( \frac{1}{n}\right)} \\
 =& \frac{-2 - \gamma}{n} + o \left( \frac{1}{n} \right) 
\end{align*}
and applying Theorem \ref{theothetatilde}, for all $\delta > 0$,
\begin{align*}
\sum_{k=1}^{n-1} \kappa_{n,k} \tau_{k+1} \gamma_{k+1}^{-1}&  \left( 1- \left(1- \tau_{k+1} \right)\frac{\tau_{k} \gamma_{k}^{-1}}{\tau_{k+1} \gamma_{k+1}^{-1}} \right)  \left\| \tilde{\theta}_{k} - \theta \right\| \\
=& O \left( \sum_{k=1}^{n} \kappa_{n,k} \tau_{k+1} \frac{(\ln k)^{1+\delta}}{k^{1-\gamma /2}} \right) \quad a.s.
\end{align*}
Then, since $\left( \frac{(\ln n)^{1/2 +\delta}}{n^{1-\gamma /2}}\right)$ is in $\mathcal{G}\mathcal{S}\left( -1+ \gamma /2 \right)$ and applying Lemma \ref{lemmok}
\[
\sum_{k=1}^{n-1} \kappa_{n,k} \tau_{k+1} \gamma_{k+1}^{-1} \left( 1- \left(1- \tau_{k+1} \right)\frac{\tau_{k} \gamma_{k}^{-1}}{\tau_{k+1} \gamma_{k+1}^{-1}} \right) \left\| \tilde{\theta}_{k} - \theta \right\| = o \left( \frac{(\ln n)^{1/2+ \delta}}{n^{1- \gamma /2}} \right) \quad a.s
\]
so that 
\[
\sum_{k=1}^{n-1} \kappa_{n,k}\tau_{k+1}\frac{ \tilde{\theta}_{k} - \theta - \left( \tilde{\theta}_{k+1} - \theta \right)}{\gamma_{k+1}}  = o \left( \frac{(\ln n)^{1/2 + \delta}}{n^{1- \gamma /2}} \right) \quad a.s.
\]

\medskip

\noindent \textbf{Bounding $\sum_{k=0}^{n-1} \kappa_{n,k}\tau_{k+1}\overline{S}_{k}^{-1} \tilde{\xi}_{k+1}$ : } First, remark that this term can be written as
\[
\sum_{k=0}^{n-1} \kappa_{n,k}\tau_{k+1}\overline{S}_{k}^{-1} \tilde{\xi}_{k+1} = \kappa_{n,0} \underbrace{\sum_{k=0}^{n-1}\prod_{j=1}^{k} \left( 1- \tau_{j} \right)^{-1} \tau_{k+1} \overline{S}_{k}^{-1} \tilde{\xi}_{k+1}}_{\overline{M}_{n}}
\]
where $\left( \overline{M}_{n} \right)$ is a martingale term with respect to the filtration $\left( \mathcal{F}_{n} \right) $, and 
\[
\left\langle M \right\rangle_{n} = \sum_{k=0}^{n-1} \left( \prod_{j=1}^{k} \left( 1- \tau_{j} \right)^{-2} \right) \tau_{k+1}^{2} \overline{S}_{k}^{-1} \mathbb{E}\left[ \tilde{\xi}_{k+1} \tilde{\xi}_{k+1}^{T} |\mathcal{F}_{k} \right] \overline{S}_{k}^{-1} ,
\]
which can be written as
\begin{align*}
\left\langle M \right\rangle_{n} & =  \sum_{k=0}^{n-1} \left( \prod_{j=1}^{k} \left( 1- \tau_{j} \right)^{-2} \right) \tau_{k+1}^{2} \overline{S}_{k}^{-1} \mathbb{E}\left[ \nabla_{h} g \left( X_{k+1},\tilde{\theta}_{k} \right)\nabla_{h} g \left( X_{k+1},\tilde{\theta}_{k} \right)^{T} |\mathcal{F}_{k} \right] \overline{S}_{k}^{-1} \\
& - \sum_{k=0}^{n-1} \left( \prod_{j=1}^{k} \left( 1- \tau_{j} \right)^{-2} \right) \tau_{k+1}^{2} \overline{S}_{k}^{-1}  \nabla G \left( \tilde{\theta}_{k} \right) \nabla G \left( \tilde{\theta}_{k} \right)^{T} \overline{S}_{k}^{-1}
\end{align*}
Since $\nabla G \left( \tilde{\theta}_{n}\right)$ and $\overline{S}_{n}^{-1}$ converge almost surely respectively to $0$ and $H^{-1}$ and applying Lemma \ref{lemmok} (third equality),
\[
\kappa_{n,0}^{2}\tau_{n}^{-1} \left\| \sum_{k=0}^{n-1} \left( \prod_{j=1}^{k} \left( 1- \tau_{j} \right)^{-2} \right) \tau_{k+1}^{2} \overline{S}_{k}^{-1}  \nabla G \left( \tilde{\theta}_{k} \right) \nabla G \left( \tilde{\theta}_{k} \right)^{T} \overline{S}_{k}^{-1} \right\| \xrightarrow[n\to + \infty]{a.s} 0 .
\]
Furthermore, since $\tilde{\theta}_{k}$ converges almost surely to $\theta$, since $\overline{S}_{k}^{-1}$ converges almost surely to $H^{-1}$ and thanks to assumption \ref{(A1c)}, there is a sequence of random matrices $R_{n}$ converging to $0$ such that
\begin{align*}
\overline{S}_{k}^{-1}&\mathbb{E}\left[ \nabla_{h} g \left( X_{k+1},\tilde{\theta}_{k} \right)\nabla_{h} g \left( X_{k+1},\tilde{\theta}_{k} \right)^{T} |\mathcal{F}_{k} \right] \overline{S}_{k}^{-1} \\
&= H^{-1}\underbrace{\mathbb{E}\left[ \nabla_{h} g \left( X, \theta \right) \nabla_{h}g \left( X , \theta \right)^{T} \right]}_{:= \Sigma} H^{-1} + R_{k} .
\end{align*}
Applying Lemma \ref{lemmok} (third equality),
\[
\kappa_{n,0}^{2} \tau_{n}^{-1} \left\| \sum_{k=0}^{n-1} \left( \prod_{j=1}^{k} \left( 1- \tau_{j} \right)^{-2} \right) \tau_{k+1}^{2} R_{k} \right\| \xrightarrow[n\to + \infty]{a.s} 0 .
\]
Finally, applying Lemma \ref{lemmok} (second equality),
\[
\kappa_{n,0}^{2}\tau_{n}^{-1} \sum_{k=0}^{n-1} \prod_{j=1}^{k} \left( 1- \tau_{j} \right)^{-2} \tau_{k+1}^{2} H^{-1}\Sigma H^{-1} \xrightarrow[n\to + \infty]{} \frac{\tau}{2\tau -1}H^{-1}\Sigma H^{-1},
\]
and then,
\[
\kappa_{n,0}^{2}\tau_{n}^{-1} \left\langle M \right\rangle_{n} \xrightarrow[n\to + \infty]{p.s} H^{-1}\Sigma H^{-1}.
\]
Then, thanks to \eqref{momenteta} and with the help of the law of large numbers for martingales \citep{Duf97}, it comes
\[
\left\| M_{n} \right\|^{2} = O \left( \ln (n) \frac{\tau_{n}}{\kappa_{n,0}^{2}} \right) \quad a.s
\]
which can be written, since $n\tau_{n}$ converges to $\tau$, as
\[
\left\| \kappa_{n,0} M_{n} \right\|^{2} = O \left( \frac{\ln n}{n} \right) \quad a.s.
\]
Furthermore, thanks to \eqref{momenteta}, the Lindeberg condition for the Central Limit Theorem for martingales is verified. Indeed, by Hölder's inequality, for all $\epsilon > 0$,
\begin{align*}
& L_{n}:= \kappa_{n,0}^{2}\tau_{n}^{-1} \sum_{k=1}^{n-1} \prod_{j=1}^{k}\kappa_{k,0}^{-2} \tau_{k+1}^{2} \mathbb{E}\left[ \left\| \overline{S}_{k}^{-1} \tilde{\xi}_{k+1} \right\|^{2}\mathbf{1}_{\kappa_{k,0}^{-1} \tau_{k+1} \left\| \overline{S}_{k}^{-1} \tilde{\xi}_{k+1} \right\| \geq \epsilon \kappa_{n,0}^{-1}\tau_{n}^{1/2}}  \Big| \mathcal{F}_{k} \right] \\
& \leq  \frac{\kappa_{n,0}^{2}}{\tau_{n}}\sum_{k=1}^{n-1} \frac{\tau_{k+1}^{2}}{\kappa_{k,0}^{2}} \left\| \overline{S}_{k}^{-1} \right\|_{op}^{2} \left( \mathbb{E}\left[ \left\|  \tilde{\xi}_{k+1}  \right\|^{2+2\eta} |\mathcal{F}_{k} \right] \right)^{\frac{1}{1+\eta}} \left( \mathbb{P} \left[\frac{\tau_{k+1}}{ \kappa_{k,0}}\left\| \overline{S}_{k}^{-1} \tilde{\xi}_{k+1} \right\| \geq \frac{\epsilon \sqrt{\tau_{n}}}{\kappa_{n,0}} |\mathcal{F}_{k} \right]  \right)^{\frac{\eta}{1+\eta}} .
\end{align*}
Then, applying Markov inequality,
\begin{align*}
L_{n} & \leq \frac{\kappa_{n,0}^{2}}{\tau_{n}}\sum_{k=0}^{n-1} \frac{\tau_{k+1}^{2}}{\kappa_{k,0}^{2}} \left\| \overline{S}_{k}^{-1} \right\|_{op}^{2} \left( \mathbb{E}\left[ \left\|  \tilde{\xi}_{k+1} \right\|^{2+2\eta} |\mathcal{F}_{k} \right] \right)^{\frac{1}{1+\eta}} \left\| \overline{S}_{k}^{-1} \right\|_{op}^{2\eta} \left( \mathbb{E}\left[ \left\| \tilde{\xi}_{k+1} \right\|^{2+2\eta} |\mathcal{F}_{k} \right] \right)^{\frac{\eta}{1+\eta}} \frac{\kappa_{n,0}^{2\eta}\tau_{k+1}^{2\eta}}{\epsilon^{2}\tau_{n}^{\eta}\kappa_{k,0}^{2\eta}} \\
& = \frac{1}{\epsilon^{2}} \frac{\kappa_{n,0}^{2+2\eta}}{\tau_{n}^{1+\eta}}\sum_{k=0}^{n-1} \frac{\tau_{k+1}^{2+2\eta}}{\kappa_{k,0}^{2+2\eta}} \left\| \overline{S}_{k}^{-1} \right\|_{op}^{2+2\eta} \mathbb{E}\left[ \left\| \tilde{\xi}_{k+1} \right\|^{2+2\eta} |\mathcal{F}_{k} \right] .
\end{align*}
Furthermore, thanks to Theorem \ref{theothetatilde}, inequality \eqref{momenteta} and {since there $\overline{S}_{n}$ converges almost surely to $H$}, there is a positive random variable $B$ such that $\left\| \overline{S}_{k}^{-1} \right\|_{op}^{2+2\eta} \mathbb{E}\left[ \left\| \tilde{\xi}_{k+1} \right\|^{2+2\eta} |\mathcal{F}_{k} \right] \leq B$ so that
\[
L_{n} \leq \frac{B}{\epsilon^{2}} \frac{\kappa_{n,0}^{2+2\eta}}{\tau_{n}^{1+\eta}}\sum_{k=0}^{n-1} \frac{\tau_{k+1}^{2+2\eta}}{\kappa_{k,0}^{2+2\eta}} 
\]
and applying Lemma \ref{lemmok}, this term converges to $0$. Then the Lindeberg condition is verified and it comes
\[
\kappa_{n,0}\tau_{n}^{-1/2} M_{n} \xrightarrow[n\to + \infty]{\mathcal{L}} \mathcal{N}\left( 0 , \frac{\tau}{2\tau-1}H^{-1}\Sigma H^{-1} \right),
\]
since $n\tau_{n}$ converges to $\tau$, this can be written as
\[
\sqrt{n} \kappa_{n,0} M_{n} \xrightarrow[n\to + \infty]{\mathcal{L}} \mathcal{N}\left( 0 , \frac{\tau^{2}}{2\tau -1} H^{-1}\Sigma H^{-1} \right) ,
\]
which concludes the proof.

\subsection{Proof of Theorem \ref{jenaipleinle}}
\label{app:proof_thm_jenaipleinle}

Let us recall that $\tilde{\theta}_{n+1}$ can be written as
\[
\tilde{\theta}_{n+1} - \theta = \tilde{\theta}_{n} - \theta - \gamma_{n+1}\overline{S}_{n}^{-1}\nabla G \left( \tilde{\theta}_{n} \right) + \gamma_{n+1}\overline{S}_{n}^{-1} \tilde{\xi}_{n+1}.
\]
Then, linearizing the gradient,
\[
\tilde{\theta}_{n+1} - \theta = \left( I_{d} - \gamma_{n+1}\overline{S}_{n}^{-1} H \right) \left(\tilde{\theta}_{n} - \theta\right) - \gamma_{n+1}\overline{S}_{n}^{-1} \tilde{\delta}_{n} + \gamma_{n+1} \overline{S}_{n}^{-1} \tilde{\xi}_{n+1},
\]
which can be written as
\begin{equation}
\tilde{\theta}_{n} - \theta = H^{-1}\overline{S}_{n}\frac{\left( \tilde{\theta}_{n} - \theta \right) - \left( \tilde{\theta}_{n+1} - \theta \right)}{\gamma_{n+1}} + H^{-1} \tilde{\delta}_{n} + H^{-1} \tilde{\xi}_{n+1}
\end{equation}
Then, thanks to decomposition \eqref{decgenmoy},
\begin{align}
\notag \theta_{n,\tau} - \theta  =&  \kappa_{n,0}\left( \theta_{0,\tau} - \theta \right) +H^{-1}\sum_{k=0}^{n-1} \kappa_{n,k}\tau_{k+1}\overline{S}_{k}\frac{\left( \tilde{\theta}_{k} - \theta \right) - \left( \tilde{\theta}_{k+1} - \theta \right)}{\gamma_{k+1}} \\
\label{decdegrosporc}&  - H^{-1}\sum_{k=0}^{n-1} \kappa_{n,k}\tau_{k+1} \tilde{\delta}_{k} + H^{-1}\sum_{k=0}^{n-1} \kappa_{n,k} \tau_{k+1} \tilde{\xi}_{k+1} . 
\end{align}
Note that the rate of convergence of the first and third terms on the right-hand side of previous equality are given in the proof of Theorem \ref{theomoy}. For the martingale term, with analogous calculus as the ones in the proof of Theorem \ref{theomoy}, one can check that
\[
\left\| H^{-1}\sum_{k=0}^{n-1} \kappa_{n,k} \tau_{k+1} \tilde{\xi}_{k+1}  \right\|^{2} = O \left( \frac{\ln n}{n} \right) \quad a.s 
\]
and
\[ \sqrt{n} \left( H^{-1}\sum_{k=0}^{n-1} \kappa_{n,k} \tau_{k+1} \tilde{\xi}_{k+1}  \right) \xrightarrow[n\to + \infty]{\mathcal{L}} \mathbb{N}\left( 0 , \frac{\tau^{2}}{2\tau -1}H^{-1}\Sigma H^{-1} \right).
\]
In order to conclude the proof, a rate of convergence for the second term on the right-hand side of equality \eqref{decdegrosporc} remains to be given. Applying an Abel's transform,
\begin{align*}
\sum_{k=0}^{n-1} \kappa_{n,k}\tau_{k+1}\overline{S}_{k}\frac{\left( \tilde{\theta}_{k} - \theta \right) - \left( \tilde{\theta}_{k+1} - \theta \right)}{\gamma_{k+1}} & = \frac{\kappa_{n,0} \tau_{1}}{\gamma_{1}}\overline{S}_{0}^{-1} \left( \tilde{\theta}_{0} - \theta \right) - \frac{\tau_{n}}{\gamma_{n}}\overline{S}_{n-1}^{-1} \left( \tilde{\theta}_{n} - \theta \right) \\
& - \underbrace{\sum_{k=1}^{n-1} \left( \kappa_{n,k-1} \frac{\tau_{k}}{\gamma_{k}}\overline{S}_{k-1} - \kappa_{n,k}\frac{\tau_{k+1}}{\gamma_{k+1}}\overline{S}_{k}\right) \left( \tilde{\theta}_{k} - \theta \right)}_{:=R_{n}}
\end{align*}
The rate of convergence of the two first term on the right hand side of previous equality are given (since $\overline{S}_{n}$ converges almost surely to $H^{-1}$) in the proof of Theorem \ref{theomoy}. Remark that since $\overline{S}_{k} = \overline{S}_{k-1} + \frac{1}{k+1}\left( \overline{u}_{k}\overline{\Phi}_{k}\overline{\Phi}_{k}^{T} - \overline{S}_{k-1} \right)$, $R_{n}$ can be rewritten as
\begin{align*}
R_{n} =& \underbrace{\sum_{k=1}^{n-1} \left( \kappa_{n,k-1} \frac{\tau_{k}}{\gamma_{k}} - \kappa_{n,k}\frac{\tau_{k+1}}{\gamma_{k+1}} \right)\overline{S}_{k-1} \left( \tilde{\theta}_{k} - \theta \right)}_{:= R_{1,n}} \\
&+ \underbrace{\sum_{k=1}^{n-1}\kappa_{n,k}\frac{\tau_{k+1}}{\gamma_{k+1}}\frac{1}{k+1} \left(  \overline{S}_{k-1} - \overline{u}_{k}\overline{\Phi}_{k}\overline{\Phi}_{k}^{T} + \frac{c_{\beta}}{k^{\beta}}Z_{k}Z_{k}^{T} \right) \left( \tilde{\theta}_{k} - \theta \right) }_{:= R_{2,n}}
\end{align*}

\noindent\textbf{Rate of convergence of $R_{1,n}$: } First, since $\kappa_{n,k-1} = \left( 1- \tau_{k} \right) \kappa_{n,k}$,
\begin{align*}
R_{1,n} &= \sum_{k=1}^{n-1} \kappa_{n,k} \left( \left( 1- \tau_{k} \right) \frac{\tau_{k}}{\gamma_{k}} - \frac{\tau_{k+1}}{\gamma_{k+1}} \right) \overline{S}_{k-1} \left( \tilde{\theta}_{k} - \theta \right)  \\
& = \sum_{k=1}^{n-1} \kappa_{n,k} \frac{\tau_{k+1}}{\gamma_{k+1}} \left( \left( 1- \tau_{k} \right) \frac{\tau_{k}\gamma_{k+1}}{\tau_{k+1}\gamma_{k}}  -1 \right) \overline{S}_{k-1}^{-1} \left( \tilde{\theta}_{k} - \theta \right)   
\end{align*}
Given that 
\[
1- \left( 1- \tau_{n-1} \right) \frac{\tau_{n-1}\gamma_{n+1}}{\tau_{n}\gamma_{n}} = \frac{2 + \gamma}{n}+ o \left( \frac{1}{n}\right),
\]
and applying Lemma \ref{lemmok} coupled with Theorem \ref{theothetatilde}, it comes that for all $\delta > 0$,
\[
\left\| R_{1,n} \right\|=  o \left( \frac{(\ln n)^{1/2 + \delta}}{n^{1-\gamma/2}} \right) \quad a.s.
\]

\noindent\textbf{Rate of convergence of $R_{2,n}$: } Thanks to Theorem \ref{theothetatilde} coupled with lemma \ref{lemmok}, one can check that for all $\delta > 0$,
\[
\left\| \tilde{\theta}_{\tau,n} - \theta \right\| = o \left( \frac{(\ln n)^{1/2 + \delta}}{n^{\gamma /2}} \right) \quad a.s.
\]
Then, let us consider the sequence of events $\left( \Omega_{n} \right)$ defined for all $n \geq 0$ by \[
\Omega_{n} = \left\lbrace \left\| \tilde{\theta}_{ n} - \theta \right\| < (\ln (n)^{1/2+\delta}\gamma_{n+1}^{1/2},\left\| \theta_{\tau,n-1} - \theta \right\| < (\ln (n))^{1/2+\delta}\gamma_{n}^{1/2} \right\rbrace . 
\]
Remark that $\mathbf{1}_{\Omega_{n}^{C}}$ converges almost surely to $0$. Then, one can write $R_{2,n}$ as
\begin{align*}
R_{2,n} = &\sum_{k=1}^{n-1}\kappa_{n,k}\frac{\tau_{k+1}}{\gamma_{k+1}}\frac{1}{k+1}\overline{S}_{k-1} \left( \tilde{\theta}_{k} - \theta \right) \\
&- \sum_{k=1}^{n}\kappa_{n,k}\frac{\tau_{k+1}}{\gamma_{k+1}}\frac{1}{k+1} \left(  \overline{u}_{k}\overline{\Phi}_{k}\overline{\Phi}_{k}^{T} + \frac{c_{\beta}}{k^{\beta}}Z_{k}Z_{k}^{T} \right) \left( \tilde{\theta}_{k} - \theta \right) \mathbf{1}_{\Omega_{k}} \\
&  - \sum_{k=1}^{n-1}\kappa_{n,k}\frac{\tau_{k+1}}{\gamma_{k+1}}\frac{1}{k+1} \left(  \overline{u}_{k}\overline{\Phi}_{k}\overline{\Phi}_{k}^{T} + \frac{c_{\beta}}{k^{\beta}}Z_{k}Z_{k}^{T} \right) \left( \tilde{\theta}_{k} - \theta \right)\mathbf{1}_{\Omega_{k}^{C}} .
\end{align*}
Applying Lemma \ref{lemmok}, for all $\delta > 0$,
\[
\left\| \sum_{k=1}^{n-1}\kappa_{n,k}\frac{\tau_{k+1}}{\gamma_{k+1}}\frac{1}{k+1}\overline{S}_{k} \left( \tilde{\theta}_{k} - \theta \right) \right\| = o \left( \frac{(\ln n)^{1/2  + \delta}}{n^{1-\gamma/2}} \right) \quad a.s.
\]
Furthermore, remark that
\begin{align*}
\sum_{k=1}^{n-1} \kappa_{n,k}\frac{\tau_{k+1}}{\gamma_{k+1}}\frac{1}{k+1} 
& \left(  \overline{u}_{k}\overline{\Phi}_{k}\overline{\Phi}_{k}^{T} + \frac{c_{\beta}}{k^{\beta}}Z_{k}Z_{k}^{T} \right) \left( \tilde{\theta}_{k} - \theta \right)  \\
&= \kappa_{n,0}\sum_{k=1}^{n-1}\kappa_{k-1,0}^{-1}\frac{\tau_{k+1}}{\gamma_{k+1}}\frac{1}{k+1} \left(   \overline{u}_{k}\overline{\Phi}_{k}\overline{\Phi}_{k}^{T} + \frac{c_{\beta}}{k^{\beta}}Z_{k}Z_{k}^{T} \right) \left( \tilde{\theta}_{k} - \theta \right)
\end{align*}
Since $\mathbf{1}_{\Omega_{n}^{C}}$ converges almost surely to $0$,
\[
\sum_{k \geq 1} \kappa_{k-1,0}^{-1} \frac{\tau_{k+1}}{\gamma_{k+1}} \frac{1}{k+1} \left\| \overline{u}_{k} \overline{\Phi}_{k}\overline{\Phi}_{k}^{T} + \frac{c_{\beta}}{k^{\beta}}Z_{k}Z_{k}^{T} \right\| \left\| \tilde{\theta}_{k} - \theta \right\| \mathbf{1}_{\Omega_{k}^{C}} < + \infty \quad a.s
\]
and
\begin{align*}
\left\| \sum_{k=1}^{n-1}\kappa_{n,k}\frac{\tau_{k+1}}{\gamma_{k+1}}\frac{1}{k+1} \left(   \overline{u}_{k}\overline{\Phi}_{k}\overline{\Phi}_{k}^{T} + \frac{c_{\beta}}{k^{\beta}}Z_{k}Z_{k}^{T} \right) \left( \tilde{\theta}_{k} - \theta \right)\mathbf{1}_{\Omega_{k}^{C}} \right\| & = O \left( \kappa_{n,0} \right) \quad a.s \\& = O \left( \frac{1}{n^{\tau}} \right) \quad a.s,
\end{align*}
which is negligible since $\tau > 1/2$. Finally, let
\begin{align*}
R_{3,n} & = \sum_{k=1}^{n-1}\kappa_{n,k}\frac{\tau_{k+1}}{\gamma_{k+1}}\frac{1}{k+1} \left\|  \overline{u}_{k}\overline{\Phi}_{k}\overline{\Phi}_{k}^{T} + \frac{c_{\beta}}{k^{\beta}}Z_{k}Z_{k}^{T} \right\| \left\|  \tilde{\theta}_{k} - \theta \right\|\mathbf{1}_{\Omega_{k}} \\
& \leq \sum_{k=1}^{n-1}\kappa_{n,k}\frac{\tau_{k+1}}{\sqrt{\gamma_{k+1}}}\frac{1}{k+1}(\ln k)^{1/2 + \delta} \left\|  \overline{u}_{k}\overline{\Phi}_{k}\overline{\Phi}_{k}^{T} + \frac{c_{\beta}}{k^{\beta}}Z_{k}Z_{k}^{T} \right\|\mathbf{1}_{\Omega_{k}}  \\
& \leq \sum_{k=1}^{n-1}\kappa_{n,k}\frac{\tau_{k+1}}{\sqrt{\gamma_{k+1}}}\frac{1}{k+1}(\ln k)^{1/2 + \delta} \underbrace{\left\|  \overline{u}_{k}\overline{\Phi}_{k}\overline{\Phi}_{k}^{T} + \frac{c_{\beta}}{k^{\beta}}Z_{k}Z_{k}^{T} \right\|\mathbf{1}_{\Omega_{k-1}'} }_{:= \Delta_{k}}
\end{align*}
with $\Omega_{k-1}' = \left\lbrace \left\| \theta_{\tau,k-1} - \theta \right\| \leq (\ln (k))^{1/2+\delta}\gamma_{k} \right\rbrace$.
Then, $R_{3,n}$ can be written as
\[
R_{3,n} = \sum_{k=1}^{n-1}\kappa_{n,k}\frac{\tau_{k+1}}{\sqrt{\gamma_{k+1}}}\frac{1}{k+1}(\ln k)^{1/2 + \delta} \mathbb{E}\left[ \Delta_{k} |\mathcal{F}_{k-1} \right] + \sum_{k=1}^{n-1}\kappa_{n,k}\frac{\tau_{k+1}}{\sqrt{\gamma_{k+1}}}\frac{1}{k+1}(\ln k)^{1/2 + \delta}\Xi_{k}
\]
with $\Xi_{k} = \Delta_{k} - \mathbb{E}\left[ \Delta_{k} |\mathcal{F}_{k-1} \right]$. Remark that $\left( \Xi_{n+1} \right)$ is a sequence of martingale differences adapted to the filtration $\left( \mathcal{F}_{n} \right)$ defined for all $n$ by $\mathcal{F}_{n} = \sigma \left( \left( X_{1},Z_{1} \right) , \ldots , \left( X_{n}, Z_{n} \right) \right)$. Thanks to inequality \eqref{majinthess} coupled with lemma \ref{lemmok},
\[
\sum_{k=1}^{n-1}\kappa_{n,k}\frac{\tau_{k+1}}{\sqrt{\gamma_{k+1}}}\frac{1}{k+1}(\ln k)^{1/2 + \delta} \mathbb{E}\left[ \Delta_{k} |\mathcal{F}_{k-1} \right] = o \left( \frac{(\ln n)^{1/2+ \delta}}{n^{1- \gamma/2}} \right) \quad a.s.
\]
Let us now consider $\alpha \in (1/2 , \tau)$, and $V_{n} = n^{2\alpha} \left( \sum_{k=1}^{n-1}\kappa_{n,k}\frac{\tau_{k+1}}{\sqrt{\gamma_{k+1}}}\frac{1}{k+1}(\ln k)^{1/2 + \delta} \Xi_{k} \right)^{2}$. Then,
\begin{align*}
\mathbb{E}\left[ V_{n} |\mathcal{F}_{n-1} \right] = \left| 1- \tau_{n} \right|^{2} \left( \frac{n}{n-1} \right)^{2\alpha} V_{n-1} + n^{2\alpha}\frac{\tau_{n}^{2}}{\gamma_{n}} \frac{(\ln (n-1))^{1+ \delta}}{n^{2}} \mathbb{E}\left[ \Delta_{n}^{2} |\mathcal{F}_{n-1} \right]
\end{align*}
Since 
\[
\left| 1- \tau_{n} \right|^{2} \left( \frac{n}{n-1} \right)^{2\alpha} = 1 - 2\frac{\tau - \alpha}{n} + o \left( \frac{1}{n} \right), 
\]
thanks to inequality \eqref{sumfini} and applying Robbins-Siegmund Theorem,
\[
 \left( \sum_{k=1}^{n-1}\kappa_{n,k}\frac{\tau_{k+1}}{\sqrt{\gamma_{k+1}}}\frac{1}{k+1}(\ln k)^{1/2 + \delta} \Xi_{k} \right)^{2} =  O\left( \frac{1}{n^{2\alpha}} \right) \quad a.s.
\]
which concludes the proof.

\subsection{Proof of Theorem \ref{theo:ASNlog}}
\label{app:proof_ASNlog}
\hfill \\

\noindent\textbf{Verifying \ref{(A1)}. } First, remark that for all $h \in \mathbb{R}^{d}$, $\left\| \nabla_{h} g (X, Y, h) \right\| \leq X$. Then, since $X$ admits a second order moment, Assumption \ref{(A1b)} is verified. Furthermore, we have
\[
\nabla G (\theta) = \mathbb{E}\left[ \nabla_{h} g (X,Y,\theta) \right] = \mathbb{E}\left[ \pi \left( \theta^{T}X \right) - Y \right] = \mathbb{E}\left[ \pi \left( \theta^{T}X \right) - \mathbb{E}[ Y | X ] \right] = 0
\]
and \ref{(A1a)} is so verified. Furthermore, since $X$ admits a second order moment and since the functional $\pi$ is continuous, the functional
\[
\Sigma : h \longmapsto \mathbb{E}\left[ \nabla_{h}g (X,Y,h) \nabla_{h}g(X,Y,h)^{T} \right] = \mathbb{E}\left[ \left( Y - \pi \left( X^{T}h \right) \right)^{2} XX^{T} \right]
\]
is continuous on $\mathbb{R}^{d}$, and in particular at $\theta$. Then, \ref{(A1c)} is verified.

\medskip

\noindent\textbf{Verifying \ref{(A2a)} and \ref{(A2b)}. } First, remark that \ref{(A2a)} is verified thanks to hypothesis. Furthermore, note that for all $h \in \mathbb{R}^{d}$, 
\[
\nabla^{2}G (h) = \mathbb{E}\left[ \pi \left( h^{T}X \right) \left( 1- \pi \left( h^{T}X \right) \right) XX^{T} \right]
\]
and by continuity of $\pi$ and since $X$ admits a second order moment, $G$ is twice continuously differentiable. Furthermore,  $\left\| \nabla^{2}G(h) \right\|_{op} \leq \frac{1}{4} \mathbb{E}\left[ \left\| X \right\|^{2} \right]$ and \ref{(A2a)} is so verified.

\medskip

\noindent\textbf{Verifying  {\ref{(H1)}}. }Only the mains ideas are given since a detailed analogous proof is available in \cite{BGBP2019} (proof of Theorem 4.1). Remark that thanks to Riccati's formula, we have
\[
\overline{S}_{n} = \frac{1}{n+1}\left( S_{0} + \sum_{k=1}^{n}a_{k}X_{k}X_{k}^{T} \right) 
\]
and that by definition, $a_{k} \geq \frac{c_{\beta}}{k^{\beta}}$. Then, 
$\lambda_{\min} \left( \overline{S}_{n} \right) \geq \frac{1}{n+1}\lambda_{\min} \left(\sum_{k=1}^{n} \frac{c_{\beta}}{k^{\beta}}   X_{k}X_{k}^{T} \right)$,
and one can easily check that
\begin{equation}\label{convaksn}
\frac{1}{\sum_{k=1}^{n}\frac{c_{\beta}}{k^{\beta}}}\sum_{k=1}^{n} \frac{c_{\beta}}{k^{\beta}} X_{k}X_{k}^{T} \xrightarrow[n\to + \infty]{p.s} \mathbb{E}\left[ XX^{T} \right]
\end{equation}
which is supposed to be positive (since $\nabla^{2}G(\theta)$ is). Then, one can easily check that
\[
\lambda_{\max} \left( \overline{S}_{n}^{-1} \right) = O \left( n^{ \beta} \right) \quad a.s.
\] 
In a same way, since $a_{k} \leq \frac{1}{4}$, one can easily check that
\[
\lambda_{\max} \left( \overline{S}_{n}^{-1} \right) = O \left( 1 \right) \quad a.s
\]
and  {\ref{(H1)}} is so verified. 

\medskip

\noindent\textbf{Conclusion 1. } Since Assumptions \ref{(A1)}, \ref{(A2a)}, \ref{(A2b)} as well as  {\ref{(H1)}} are verified, Theorem \ref{theothetanbar} holds, i.e $\tilde{\theta}_{n}$ and $\theta_{\tau,n}$ converge almost surely to $\theta$. 

\medskip

\noindent\textbf{Verifying  {the strong consistency of $\overline{S}_{n}$} }Only the mains ideas are given since a detailed analogous proof is available in \cite{BGBP2019} (proof of Theorem 4.1). Remark that $\overline{S}_{n}$ can be written as
\[
\overline{S}_{n} = \frac{1}{n+1} \left( \overline{S}_{0} + \sum_{k=1}^{n} \overline{a}_{k}X_{k}X_{k}^{T} + \sum_{k=1}^{n} \left( a_{k} - \overline{a}_{k} \right) X_{k}X_{k}^{T} \right) 
\] 
and since $a_{k} - \overline{a}_{k} \neq 0$ if and only if $a_{k} > \overline{a}_{k}$, it comes thanks to equation \eqref{convaksn},
\begin{align*}
\left\| \frac{1}{n+1}\overline{S}_{0} + \frac{1}{n+1}\sum_{k=1}^{n}a_{k}X_{k}X_{k}^{T} \right\|_{op} &=  \frac{1}{n+1}\left\| \overline{S}_{0} \right\|_{op} + \frac{1}{n+1}\left\| \sum_{k=1}^{n} \frac{c_{\beta}}{k^{\beta}}X_{k}X_{k}^{T} \right\|_{op} \\
&= O \left( n^{- \beta} \right) \quad a.s.
\end{align*}
Furthermore, as in the proof of Theorem 4.1 in \cite{BGBP2019}, one can check, since $\theta_{\tau, n}$ converges to $\theta$,   that
\[
\frac{1}{n}\sum_{k=1}^{n}\overline{a}_{k} X_{k}X_{k}^{T} = \frac{1}{n}\sum_{k=1}^{n} \pi \left( \theta_{\tau,k-1}^{T} X_{k}\right) \left( 1- \pi \left( \theta_{\tau,k-1}^{T} X_{k}\right) \right) X_{k}X_{k}^{T} \xrightarrow[n\to + \infty]{a.s} \nabla^{2}G(\theta) .
\]
and the consistency of $\overline{S}_{n}$ is therefore verified.

\medskip

\noindent\textbf{Conclusion 2.} Since Assumptions \ref{(A1)}, \ref{(A2a)}, \ref{(A2b)}, { and \ref{(H1)} as well as the consistency of $\overline{S}_{n}$ are verified},  if $X$ admits a moment of order $2+2\eta$ with $\eta > \frac{1}{\alpha}-1$, Theorem \ref{theothetatilde} holds, i.e
\[
\left\| \tilde{\theta}_{n} - \theta \right\|^{2} = O \left( \frac{\ln n}{n^{\gamma}} \right) \quad a.s .
\]

\medskip

\noindent\textbf{Verifying \ref{(A2c)}. } Thanks to Lemma 6.2 in \cite{BGBP2019}, we have for all $h_{1},h_{2} \in \mathbb{R}^{d}$, 
\[
\left| \pi \left( h_{1}^{T}X \right) \left( 1- \pi  \left( h_{1}^{T}X \right) \right) -\pi \left( h_{2}^{T}X \right) \left( 1- \pi  \left( h_{2}^{T}X \right) \right) \right| \leq \frac{1}{12\sqrt{3}} \left\| X \right\| \left\| h_{1} - h_{2} \right\|
\]
and in a particular case, it comes
\[
\left\| \nabla^{2}G\left( h_{1} \right) - \nabla^{2}G\left( h_{2} \right) \right\|_{op} \leq \frac{1}{12\sqrt{3}}\mathbb{E}\left[ \left\| X \right\|^{3} \right] \left\| h_{1 } - h_{2} \right\|
\]
and \ref{(A2c)} is verified since $X$ admits a third order moment. 

\medskip

\noindent\textbf{Verifying inequalities \eqref{majinthess} and \eqref{sumfini}. } Remark that for all $n \geq 1$, 
\[
\left\| a_{n} X_{n}X_{n}^{T} \right\| \leq \max \left\lbrace \frac{1}{4} , c_{\beta} \right\rbrace  \left\| X \right\|^{2} =: C_{a,\beta} \| X \|^{2}  
\]	
so that, if $X$ admits a fourth order moment,
\[
\mathbb{E}\left[ \left\| a_{n} X_{n}X_{n}^{T} \right\| \right] \leq C_{a,\beta}\mathbb{E}\left[ \left\| X \right\|^{2} \right] \quad \quad \text{and} \quad \quad \mathbb{E}\left[ \left\| a_{n} X_{n}X_{n}^{T} \right\|^{2} \right] \leq C_{a,\beta}^{2} \mathbb{E}\left[ \left\| X \right\|^{4} \right] 
\]
and inequalities \eqref{majinthess} and \eqref{sumfini} are so verified.

\medskip

\noindent\textbf{Conclusion 3. } Theorem \ref{jenaipleinle} holds, meaning that
\[
\left\| \theta_{\tau,n} - \theta \right\|^{2} = O \left( \frac{\ln n}{n} \right) \quad a.s 
 \quad \quad \text{and} \quad \quad 
\sqrt{n} \left( \theta_{\tau, n } - \theta \right) \xrightarrow[n\to + \infty]{\mathcal{L}} \mathcal{N}\left( 0 , H^{-1} \right) .
\]

\noindent\textbf{Convergence of $\overline{S}_{n}$. } First, let us recall 
\[
\overline{S}_{n} = \frac{1}{n+1} \left( \overline{S}_{0} + \underbrace{\sum_{k=1}^{n} \overline{a}_{k}X_{k}X_{k}^{T}}_{A_{n}} + \sum_{k=1}^{n} \left( a_{k} - \overline{a}_{k} \right) X_{k}X_{k}^{T} \right) 
\]
and that
\[
\left\| \frac{1}{n+1} \left( \overline{S}_{0} + \sum_{k=1}^{n} \left( a_{k} - \overline{a}_{k} \right) X_{k}X_{k}^{T} \right) \right\|^{2} = O \left( \frac{1}{n^{2\beta}} \right) \quad a.s.
\]
Furthermore, let us split $A_{n}$ into two terms, i.e
\[
A_{n} =\sum_{k=1}^{n} \nabla^{2}G \left( \theta_{\tau,k-1} \right) + \sum_{k=1}^{n} \Xi_{k}
\]
with $\Xi_{k} = \overline{a}_{k}X_{k}X_{k}^{T} - \nabla^{2}G \left( \theta_{\tau,k-1} \right) $. Since $\left( \xi_{n} \right)$ is a sequence of martingale differences adapted to the filtration $\left( \mathcal{F}_{n} \right)$ and since $\mathbb{E}\left[ \left\| \Xi_{k} \right\|^{2} |\mathcal{F}_{k-1} \right] \leq \frac{1}{16}\mathbb{E}\left[ \left\| X \right\|^{4} \right]$, we have (see Theorem 6.2 in \cite{CGBP2020}) for all $\delta > 0$
\[
\left\| \frac{1}{n+1}\sum_{k=1}^{n} \Xi_{k} \right\|^{2} = o \left( \frac{(\ln n)^{1+\delta}}{n} \right) \quad a.s.
\]
Furthermore, since $X$ admits a third order moment, the Hessian is $\frac{1}{12\sqrt{3}}\mathbb{E}\left[ \| X \|^{3} \right]$-lipschitz and for all $\delta > 0$, by Toeplitz Lemma,
\begin{align*}
\left\| \frac{1}{n}\sum_{k=1}^{n} \nabla^{2}G \left( \theta_{\tau,k-1} \right) - \nabla^{2}G(\theta) \right\| & \leq \frac{\mathbb{E}\left[ \| X \|^{3} \right]}{12\sqrt{3}n}\sum_{k=1}^{n} \left\| \theta_{\tau,k-1} - \tau_{k} \right\| \\
&= o \left( \frac{(\ln n)^{1/2+\delta/2}}{\sqrt{n}} \right) \quad a.s,
\end{align*}
which concludes the proof.

\subsection{Proofs of Theorems \ref{theo:ASNsoft} and \ref{ASNsoft}}
In what follows, let us recall that we consider the functional $G : \mathbb{R}^{p} \times \ldots \times \mathbb{R}^{p} \longrightarrow \mathbb{R}$ defined for all $h = \left( h_{1} , \ldots , h_{K} \right)$ by
\[
G(h) := \mathbb{E}\left[ \log \left( \frac{e^{h_{Y}^{T}X}}{ \sum_{k=1}^{K}e^{h_{k}^{T}X}} \right) \right]
\]

\subsubsection{Some results on the functional $G$}
\noindent\textbf{Verifying assumption \ref{(A1)}.} First remark that  
\[
\left\| \nabla_{h} g \left( X,Y,h \right) \right\| \leq \sqrt{K}\left\| X \right\| .
\]
and if $X$ admits a second order moment, \ref{(A1b)} is verified.
Furthermore it is evident that \ref{(A1a)} is verified. Indeed, we have
\[
\nabla G (\theta) = \mathbb{E}\left[ \mathbb{E}\left[ \nabla_{h} g (X,Y,\theta) | X \right] \right] = \mathbb{E}\left[ \begin{pmatrix}
X \left( \frac{e^{\theta_{1}^{T}X}}{\sum_{k=1}^{K}e^{\theta_{k}^{T}X}} - \mathbb{E}\left[ \mathbf{1}_{Y=1} | X \right] \right) \\
\vdots \\
X \left( \frac{e^{\theta_{K}^{T}X}}{\sum_{k=1}^{K}e^{\theta_{k}^{T}X}} - \mathbb{E}\left[ \mathbf{1}_{Y=K} | X \right] \right)
\end{pmatrix} \right] = 0 .
\]

Furthermore, for all $h$,
\begin{align}
\notag\mathbb{E} & \left[ \nabla_{h} g (X,Y , h )\nabla_{h}g \left( X,Y,h \right)^{T} \right] - \mathbb{E}\left[ \nabla_{h} g (X,Y , \theta )\nabla_{h}g \left( X,Y,\theta \right)^{T} \right]  \\
\label{maremaremare}& = \mathbb{E}\left[ \left( \sigma (X , h ) - \sigma (X,\theta)\right)\left( \sigma (X , h ) - \sigma (X,\theta)\right)^{T} \otimes XX^{T} \right]  .
\end{align}
and since the functional $\sigma$ is bounded and continuous, by dominated convergence, since $X$ admits a second order moment, \ref{(A1c)} is verified.

\medskip

\noindent\textbf{Verifying assumption \ref{(A2)}.} First, remark that for all $h$, since $\left\| \sigma (.,.) \right\|$ is bounded by $\sqrt{K}$,
\begin{align*}
\left\| \nabla^{2} G(h) \right\|_{op} \leq \mathbb{E}\left[ \left\| X \right\|^{2} \left\|\sigma \left( X , h \right) \right\| \right] \leq \sqrt{K}\mathbb{E}\left[ \left\| X \right\|^{2} \right] .
\end{align*}
Then, if $X$ admits a second order moment, assumption \ref{(A2a)} is verified. Furthermore, \ref{(A2b)} is verified by hypothesis. Finally, let us consider the functional $F_{k'} : [0,1] \longrightarrow \mathbb{R}$ defined for all $t$ by
\[
F_{k'}(t) = \frac{e^{\left(\theta_{k'}+t\left(h_{k'} - \theta_{k'}\right)\right)^{T}X}}{\sum_{k=1}^{K}e^{\left( \theta_{k} + t \left( h_{k} - \theta_{k} \right) \right)^{T}X}}
\]
Then, for all $t \in [0,1]$,
\[
F'(t) = \frac{\left(h_{k'} - \theta_{k'} \right)^{T}Xe^{\left(\theta_{k'}+t\left(h_{k'} - \theta_{k'}\right)\right)^{T}X}}{\sum_{k=1}^{K}e^{\left( \theta_{k} + t \left( h_{k} - \theta_{k} \right) \right)^{T}X}} - \underbrace{ \frac{e^{\left( \theta_{k'} +  \left( h_{k'} - \theta_{k'} \right) \right)^{T}X} \sum_{k=1}^{K}t\left( h_{k} - \theta_{k} \right)^{T}Xe^{\left( \theta_{k} + t \left( h_{k} - \theta_{k} \right) \right)^{T}X}}{\left( \sum_{k=1}^{K}e^{\left( \theta_{k} + t \left( h_{k} - \theta_{k} \right) \right)^{T}X} \right)^{2}}}_{(*)}
\]
First, one can check that for all $t \in [0,1]$,
\[
\left| \frac{\left(h_{k'} - \theta_{k'} \right)^{T}Xe^{\left(\theta_{k'}+t\left(h_{k'} - \theta_{k'}\right)\right)^{T}X}}{\sum_{k=1}^{K}e^{\left( \theta_{k} + t \left( h_{k} - \theta_{k} \right) \right)^{T}X}} \right| \leq \left\| h_{k'} - \theta_{k'} \right\| \| X \| .
\]
Furthermore, applying Cauchy-Schwarz inequality,
\begin{align*}
(*) & \leq \left\| X \right\| \frac{e^{\left( \theta_{k'} + t \left( h_{k'} - \theta_{k'} \right) \right)^{T}X}}{\left( \sum_{k=1}^{K}e^{\left( \theta_{k} + t \left( h_{k} - \theta_{k} \right) \right)^{T}X} \right)^{2}}\sqrt{\sum_{k=1}^{K}\left\| h_{k} - \theta_{k} \right\|^{2}} \sqrt{\sum_{k=1}^{K}e^{2\left( \theta_{k} + t \left( h_{k} - \theta_{k} \right) \right)^{T}X}} \\
&\leq   \| X \| \left\| \theta - h \right\| .
\end{align*}
Then,  
\begin{equation}
\label{lipver} \left| \frac{e^{\theta_{k'}^{T}X}}{\sum_{k=1}^{K}e^{\theta_{k}^{T}X}} - \frac{e^{h_{k'}^{T}X}}{\sum_{k=1}^{K}e^{h_{k}^{T}X}} \right| \leq  \left\| X \right\| \left( \left\| h_{k'} - \theta_{k'} \right\| + \left\| \theta - h \right\| \right) .
\end{equation}
Then,
\begin{equation}\label{maremaremare2}
\left\| \sigma (X,h ) - \sigma(X,\theta) \right\| \leq 2\sqrt{K} \left\| X \right\| \left\| \theta - h \right\|  
\end{equation}
and
\[
\left\| \text{diag} \left( \sigma (X,\theta)\right) - \text{diag} \left( \sigma (X, h ) \right) \right\| \leq 2\sqrt{K} \left\| X \right\| \left\| \theta - h \right\|  
\]
Then,
\[
\left\| \nabla^{2}G(\theta) - \nabla^{2}G(h) \right\| \leq 6\sqrt{K}\mathbb{E}\left[ \left\| X \right\|^{3}\right] \left\| h - \theta \right\|.
\]
Finally, if $X$ admits a third order moment, the Hessian is $6\sqrt{K} \mathbb{E}\left[ \left\| X \right\|^{3} \right]$-Lipschitz and Assumption \ref{(A2c)} is thus verified.

\subsubsection{Proof of Theorems \ref{theo:ASNsoft} and \ref{ASNsoft}}
\label{proof:ASNsoft}

\begin{proof}[Proof of Theorem  \ref{theo:ASNsoft}] \hfill \\

\noindent\textbf{Verifying \ref{(H1)}. } Remark that
\[
\lambda_{\min}\left( H_{n} \right) \geq \lambda_{\min} \left( \sum_{k=1}^{n} \beta_{k}Z_{k}Z_{k}^{T} \right)
\]
and that, since $\beta\in (0,1/2)$,
\[
\frac{1- \beta}{c_{\beta}^{-1}n^{1-\beta}} \sum_{k=1}^{n} \beta_{k} Z_{k}Z_{k}^{T} \xrightarrow[n\to + \infty]{a.s} I_{pK} .
\]
Then,
\[
\left\| \overline{H}_{n}^{-1} \right\| = O \left( n^{\beta} \right) \quad a.s.
\]
Furthermore, since $\left\| \nabla_{h} g (X,Y,h) \right\| \leq \left\| X \right\|$, and {since $X$ admits a second order moment},
\[
\frac{1}{n+1}\left\| \sum_{k=1}^{n} \nabla_{h}g \left( X_{k} , Y_{k} , \theta_{k-1} \right)\nabla_{h}g \left( X_{k} , Y_{k} , \theta_{k-1} \right)^{T} \right\| \leq \frac{1}{n}\sum_{k=1}^{n} \left\| X_{k} \right\|^{2} \xrightarrow[n\to + \infty]{a.s} \mathbb{E}\left[ \left\| X \right\|^{2} \right]
\]
ans assumption \ref{(H1)} is so verified. 

\medskip

\noindent\textbf{Conclusion 1. }Since Assumptions \ref{(A1a)}, \ref{(A1b)},\ref{(A2a)}, \ref{(A2b)} and \ref{(H1)} are fulfilled, Theorem \ref{snaas} holds, i.e $\theta_{n}$ converges almost surely to $\theta$. 

\medskip

\noindent\textbf{Verifying \ref{(H2a)}. } First, let us rewrite $\overline{H}_{n}$ as
\[
\overline{H}_{n} = \frac{1}{n+1} \left( \overline{H}_{0}^{-1} + \underbrace{\sum_{k=1}^{n}\nabla_{h} g \left( X_{k} , Y_{k} , \theta_{k-1} \right) \nabla_{h}g \left( X_{k} , Y_{k} , \theta \right)^{T}}_{:= A_{n}} + \sum_{k=1}^{n} \frac{c_{\beta}}{k^{\beta}}Z_{k}Z_{k}^{T} \right) 
\]
and we have already proven that
\[
\left\| \frac{1}{n+1}  \left( \overline{H}_{0}^{-1} + \sum_{k=1}^{n} \frac{c_{\beta}}{k^{\beta}}Z_{k}Z_{k}^{T} \right) \right\|^{2} = O \left( \frac{1}{n^{2\beta}} \right) \quad p.s.
\]
Furthermore, one can rewrite $A_{n}$ as
\[
A_{n} = \sum_{k=1}^{n} \mathbb{E}\left[ \nabla_{h} g \left( X_{k} , Y_{k} , \theta_{k-1} \right) \nabla_{h}g \left( X_{k} , Y_{k} , \theta_{k-1} \right)^{T} |\mathcal{F}_{k-1} \right] + \sum_{k=1}^{n} \Xi_{k}
\]
with
\begin{align*}
\Xi_{k} :=& \nabla_{h} g \left( X_{k} , Y_{k} , \theta_{k-1} \right) \nabla_{h}g \left( X_{k} , Y_{k} , \theta_{k-1} \right)^{T} \\
&- \mathbb{E}\left[ \nabla_{h} g \left( X_{k} , Y_{k} , \theta_{k-1} \right) \nabla_{h}g \left( X_{k} , Y_{k} , \theta_{k-1} \right)^{T} |`\mathcal{F}_{k-1} \right].    
\end{align*} 
First, note that if $X$ admits a fourth order moment, $\mathbb{E}\left[ \left\| \nabla_{h}g \left( X_{k} , Y_{k} , \theta_{k-1} \right) \right\|^{4} |\mathcal{F}_{k-1} \right] \leq K^{2} \mathbb{E}\left[ \left\| X \right\|^{4} \right] $, so that applying Theorem 6.2 in \cite{CGBP2020}, for all $\delta > 0 $,
\[
\left\| \frac{1}{n+1}\sum_{k=0}^{n} \Xi_{k} \right\|^{2} = o \left( \frac{(\ln n)^{1+\delta}}{n} \right) \quad a.s.
\]
Furthermore, we have proven that if $X$ admits a second order moment in the last subsection, \ref{(A1c)} is fulfilled. Then, since $\theta_{n}$ converges almost surely to $\theta$ and by continuity
\begin{align*}
\frac{1}{n+1}\sum_{k=1}^{n} \mathbb{E}\left[\nabla_{h}g \left( X_{k} , Y_{k} , \theta_{k-1} \right)\nabla_{h}g \left( X_{k} , Y_{k} , \theta_{k-1} \right)^{T} |\mathcal{F}_{k-1} \right] \\
\xrightarrow[n\to + \infty]{a.s} \mathbb{E}\left[ \nabla_{h} g \left( X, Y , \theta \right) \nabla_{h} g \left( X,Y, \theta \right) \right] = H,
\end{align*}
and Assumption \ref{(H2a)} is so fulfilled.

\medskip

\noindent\textbf{Conclusion 2. } If $X$ admits a second order moment, Assumptions \ref{(A1)}, \ref{(A2)}, \ref{(H1)}, \ref{(H2a)} are verified and Theorem \ref{theoratetheta} holds, i.e
\[
\left\| \theta_{n} - \theta \right\|^{2} = O \left( \frac{\ln n}{n} \right) \quad a.s.
\]

\medskip

\noindent\textbf{Verifying (H2b). } Note that thanks to inequalities \eqref{maremaremare} and \eqref{maremaremare2},
\begin{align*}
\bigg\| \mathbb{E}\big[ \nabla_{h}g \left( X, Y , h \right) \nabla_{h}g \left( X , Y , h \right)  &- \nabla_{h} g \left( X , Y , \theta \right) \nabla_{h} \left( X, Y ,\theta \right)^{T}  \big] \bigg\| \\
&  \leq \mathbb{E}\left[ \left\| \sigma (X , h ) - \sigma (X , \theta ) \right\|^{2} \left\| X \right\|^{2} \right] \\
& \leq 4K \mathbb{E}\left[ \left\| X \right\|^{4} \right] \left\| h- \theta \right\|^{2}.
\end{align*} 
Then, thanks to the Toeplitz lemma, it comes that for all $\delta > 0$,
\begin{align*}
\bigg\| \frac{1}{n+1}\sum_{k=1}^{n}\mathbb{E} \bigg[ \nabla_{h} g \left( X_{k} , Y_{k} , \theta_{k-1} \right) & \nabla_{h}g \left( X_{k} , Y_{k} , \theta_{k-1} \right)^{T}|\mathcal{F}_{k-1} \bigg] - H \bigg\| \\
&  \leq \frac{\left\| H   \right\|}{n+1} + \frac{4K\mathbb{E}\left[ \left\| X \right\|^{4} \right]}{n+1} \sum_{k=1}^{n} \left\| \theta_{k-1} - \theta \right\|^{2}  \\
& = o \left( \frac{(\ln n)^{1+\delta}}{n} \right) \quad a.s.
\end{align*}
Then, 
\[
\left\| \overline{H}_{n} - H \right\|^{2} = O \left( \frac{1}{n^{2\beta}} \right) \quad a.s,
\]
and \ref{(H2b)} is so verified.

\medskip

\noindent\textbf{Conclusion 3. } If $X$ admits a fourth order moment, Assumptions \ref{(A1)}, \ref{(A2)}, \ref{(H1)}, \ref{(H2a)} and \ref{(H2b)} are fulfilled, so that Theorem \ref{thetatlc} holds. Since 
\[
\mathbb{E}\left[ \nabla_{h} g \left( X, Y , \theta \right) \nabla_{h} g \left( X, Y ,\theta \right)^{T} \right] = H,
\]
we thus have
\[
\sqrt{n} \left( \theta_{n} - \theta \right) \xrightarrow[n \to + \infty]{\mathcal{L}} \mathcal{N}\left( 0 , H^{-1} \right) .
\]

\end{proof}

\begin{proof}[Proof of Theorem \ref{ASNsoft}] \hfill \\

\noindent\textbf{Verifying  {\ref{(H1)}} and conclusion 1. }If $X$ admits a second order moment, with the same calculus as in the proof of Theorem \ref{theo:ASNsoft}, up to take $\beta < \gamma  -1/2$ instead of $\beta < 1/2$, one can check that Assumption  {\ref{(H1)}} is verified  {with $\beta < \gamma -1/2$}. Then, $\tilde{\theta}_{n}$ and $\theta_{\tau,n}$ converge almost surely to $\theta$.

\medskip

\noindent\textbf{Verifying  { the consistency of $\overline{S}_{n}$} and conclusion 2. } If $X$ admits a fourth order moment, with the same calculus as in the proof of Theorem \ref{theo:ASNsoft}, up to take $\beta < \gamma  -1/2$ instead of $\beta < 1/2$, one can check that  { the strong consistency of $\overline{S}_{n}$} is verified. Then, 
\[
\left\| \tilde{\theta}_{n} - \theta \right\|^{2} = O \left( \frac{\ln n}{n^{\gamma}} \right) \quad a.s.
\]
Furthermore, let us recall that
\[
\theta_{\tau,n} - \theta  = \kappa_{n,0}\left( \theta_{0,\tau} - \theta \right) + \sum_{k=0}^{n-1}\kappa_{n,k} \tau_{k+1} \left( \tilde{\theta}_{k} - \theta \right)
\]
and that $\left| \kappa_{n,0} \right| = O \left( n^{-\tau} \right)$. Furthermore, applying Lemma \ref{lemmok}, for all $\delta > 0$,
\[
\sum_{k=0}^{n-1} \kappa_{n,k} \tau_{k+1} \left\| \tilde{\theta}_{k} - \theta \right\| = o \left( \frac{(\ln n)^{1/2+\delta/2}}{n^{\gamma /2}} \right) \quad a.s
\]
leading, since $\tau > 1/2$, to
\[
\left\| \theta_{\tau,n} - \theta \right\|^{2} = o \left( \frac{(\ln n)^{1+\delta}}{n^{\gamma}} \right) \quad a.s.
\] 

\medskip

\noindent\textbf{Verifying \ref{(H2b)}. } Note that by inequalities \eqref{maremaremare} and \eqref{maremaremare2},
\begin{align*}
\bigg\| \mathbb{E}\bigg[ \nabla_{h}g \left( X, Y , h \right) \nabla_{h}g \left( X , Y , h \right) & - \nabla_{h} g \left( X , Y , \theta \right) \nabla_{h} \left( X, Y,\theta \right)^{T}  \bigg] \bigg\| \\
&  \leq \mathbb{E}\left[ \left\| \sigma (X , h ) - \sigma (X , \theta ) \right\|^{2} \left\| X \right\|^{2} \right] \\
& \leq 4K \mathbb{E}\left[ \left\| X \right\|^{4} \right] \left\| h- \theta \right\|^{2}.
\end{align*} 
Then, thanks to the Toeplitz lemma, it comes that for all $\delta > 0$,
\begin{align*}
&  \left\| \frac{1}{n+1}\sum_{k=1}^{n}\mathbb{E}\left[ \nabla_{h} g \left( X_{k} , Y_{k} , \theta_{\tau,k-1} \right) \nabla_{h}g \left( X_{k} , Y_{k} , \theta_{\tau , k-1} \right)^{T}|\mathcal{F}_{k-1} \right] - H \right\| \\
&   \leq \frac{\left\| H   \right\|}{n+1} + \frac{4K\mathbb{E}\left[ \left\| X \right\|^{4} \right]}{n+1} \sum_{k=1}^{n} \left\| \theta_{\tau , k-1} - \theta \right\|^{2}  
 \\
& = o \left( \frac{(\ln n)^{1+\delta}}{n^{\gamma}} \right) \quad a.s.
\end{align*}
Then, since $\beta < \gamma - 1/2$,
\[
\left\| \overline{H}_{n} - H \right\|^{2} = O \left( \frac{1}{n^{2\beta}} \right) \quad a.s,
\]
and \ref{(H2b)} is so verified.

\medskip

\noindent\textbf{Conclusion 3. } If $X$ admits a fourth order moment, Assumptions \ref{(A1)}, \ref{(A2)}, { and \ref{(H2b)}} are fulfilled, and Theorem \ref{theomoy} holds, i.e
\[
\left\| \theta_{\tau,n} - \theta \right\| = O \left( \frac{\ln n}{n} \right) \quad a.s \quad \quad \text{and} \quad \quad \sqrt{n} \left( \theta_{\tau,n} - \theta \right) \xrightarrow[n\to + \infty]{\mathcal{L}} \mathcal{N}\left( 0 , H^{-1} \right)
\]

\subsection{A useful technical Lemma}
Let us now give an useful technical lemma introduced (as Lemma 4) in \cite{mokkadem2011generalization}.
\begin{lem}\label{lemmok}
Let $\left( a_{n} \right) \in \mathcal{G}\mathcal{S}\left( a^{*} \right)$ with $a^{*} >0 $ and $na_{n} \xrightarrow[n\to + \infty]{} \psi \geq 0$. Let $m > 0$ and $\left( u_{n} \right) \in \mathcal{G}\mathcal{S}\left( u^{*} \right)$ with $u^{*}  $ such that $m + u^{*}\psi > 0$ and $\alpha_{n}$ such that $\alpha_{n}\xrightarrow[n\to + \infty]{} 0 $. Then
\begin{align*}
    & u_{n}^{-1} \prod_{i=1}^{n} \left( 1- a_{j} \right)^{m} \xrightarrow[n\to + \infty]{} 0 . \\
    & u_{n}^{-1}\sum_{k=1}^{n} \prod_{j=k+1}^{n} \left( 1- a_{j} \right)^{m}a_{k}u_{k} \xrightarrow[n\to + \infty]{} \left( m + u^{*}\psi \right)^{-1} .\\
    & u_{n}^{-1}\sum_{k=1}^{n} \prod_{j=k+1}^{n} \left( 1- a_{j} \right)^{m}a_{k}u_{k}\alpha_{k}\xrightarrow[n\to + \infty]{} 0.
\end{align*}
\end{lem}

\end{proof}



\end{document}